\documentclass[a4paper]{article}
\usepackage{amsmath}
\usepackage{amsthm}
\usepackage{amssymb}
\usepackage{graphicx}
\usepackage{authblk}
\usepackage{hyperref}
\usepackage{multirow}
\usepackage{cite}
\usepackage{comment}
\usepackage{enumitem}

\newtheorem{theorem}{Theorem}[section]
\newtheorem{proposition}[theorem]{Proposition}
\newtheorem{lemma}[theorem]{Lemma}

\newtheorem{remark}{Remark}

\DeclareMathOperator{\cycl}{Cycl}

\newcommand{\includegraph}[2][]{\ifnum\pdfoutput=0\includegraphics[#1]{#2.eps}\else\includegraphics[#1]{#2.pdf}\fi}

\newcommand{\propref}[1]{Proposition~\ref{proposition:#1}}
\newcommand{\proplab}[1]{\label{proposition:#1}}

\author[1]{Renato Huzak}

\author[2]{Kristian Uldall Kristiansen}

\affil[1]{Hasselt University, Campus Diepenbeek, Agoralaan Gebouw D, 3590 Diepenbeek, Belgium}
\affil[2]{Department of Applied Mathematics and Computer Science, Technical University of Denmark, 2800 Kgs. Lyngby, Denmark}

\title{Sliding cycles of the regularized piecewise linear $VI_3$ two-fold}

\date{}

\begin{document}
\maketitle

\begin{abstract}
The goal of this paper is to study the number of sliding limit cycles of a regularized piecewise linear
$VI_3$ two-fold using the notion of slow divergence integral. We focus on limit cycles produced by canard cycles located in the half-plane with an invisible fold point. We prove that the integral has at most $1$ zero counting multiplicity (when it is not identically zero). This will imply that the canard cycles can produce at most $2$ limit cycles. Moreover, we detect regions in the parameter space with $2$ limit cycles.
\end{abstract}

\tableofcontents

\section{Introduction}\label{Section-Introduction}
In this paper, we consider the regularization of piecewise smooth linear systems:
\begin{align}
 \dot z &= Z^+(z)\phi(h(z)\epsilon^{-2})+Z^-(z)(1- \phi(h(z)\epsilon^{-2})),\quad z\in \mathbb R^2, \label{Zreg}
\end{align}
with $Z^+$ and $Z^-$ planar affine vector-fields, and $h:\mathbb R^2\rightarrow \mathbb R$ linear and $\nabla h\ne 0$. Finally, $\phi:\mathbb R\rightarrow \mathbb R$ is a regularization function:
\begin{align}
 \phi'(s)>0\mbox{ for all }  s\in \mathbb R,
\quad 
 \phi(s) \rightarrow \begin{cases}
                      1 &\text{for}\,\,s\rightarrow \infty\\
                      0 &\text{for}\,\,s\rightarrow -\infty
                     \end{cases}\label{phi1st}
\end{align}
Regularized piecewise smooth systems have received a great deal of attention in the recent years, see \cite{bossolini2020a,jelbart2021c,jelbart2021b,kosiuk2016a,kristiansen2018a,kristiansen2020a,kristiansen2015a,RHKK} and references therein. In fact, piecewise smooth (PWS, henceforth) systems 
\begin{align}
 \dot z &=\begin{cases}
           Z^+(z)  \text{ for }h(z)>0,\\
           Z^-(z)  \text{ for }h(z)<0,
          \end{cases}\label{pws}
\end{align}
corresponding by (\ref{phi1st}) to the singular limit $\epsilon\rightarrow 0$ of (\ref{Zreg}), occur naturally in many different applications, including in models of friction (see \cite{berger2002a}). 
 If $Z^\pm$ and $h$ are each affine, then (\ref{pws}) is said to be piecewise linear (PWL). 


The interest in (\ref{Zreg}) for $0<\epsilon\ll 1$ is partially motivated from the desire to understand how piecewise smooth phenomena (folds, grazing, boundary equilibria \cite{Kuznetsov2003}) unfold in the smooth version \cite{jelbart2021b,jelbart2021c,kristiansen2018a,kristiansen2015a,bonet-reves2018a}. For this purpose methods from Geometric Singular Perturbation Theory (GSPT) and blowup have been refined to deal with resolving the special singular limit of  (\ref{Zreg}), see \cite{RHKK,kristiansen2018a,kristiansen2020a}.

In \cite{RHKK}, the present authors showed that the number of limit cycles of (\ref{Zreg}) for $\epsilon>0$ is unbounded when $Z^\pm$ are quadratic vector-fields. In particular, we showed that there exist quadratic vector-fields $Z^\pm(\cdot,\lambda)$, depending smoothly on a parameter $\lambda$, such that the following statement is true:  For any $k\in \mathbb N$ there exist a regularization function $\phi_k$ satisfying (\ref{phi1st}) and a continuous function $\lambda_k:[0,\epsilon_k[\rightarrow \mathbb R$, with $\epsilon_k>0$, such that (\ref{Zreg}) with $Z^\pm(\cdot,\lambda_k(\epsilon))$ and $\phi=\phi_k$ has at least $k+1$ limit cycles. The limit cycles were constructed through $k$ simple zeros of a slow divergence integral that was associated with the so-called $VI_3$ two-fold of PWS systems.  In \cite{RHKKGR2023}, the notion of slow divergence integrals in the context of  regularized PWS systems is developed further, but historically slow divergence integrals were developed by De Maesschalck, Dumortier and Roussarie, see \cite{1996,DM-entryexit,DR2007,DM,DDR-book-SF} and references therein, as a tool in slow-fast systems and canard theory to detect limit cycles. In particular, the roots of the slow divergence integral provide candidates for limit cycles. For example, this tool can be used to find good lower bounds on the number of limit cycles in slow-fast Li\'{e}nard equations, see \cite{DDMoreLC,SDICLE1,DPR,lvarez2020a}.
For PWS systems, the slow motion along slow manifolds of slow-fast systems is replaced by sliding (following Filippov \cite{filippov1988differential}) along subsets of the discontinuity set $\Sigma$ where $Z^\pm$ are in opposition relative to $\Sigma$, see Fig.  \ref{fig:2examples}. Closed orbits with sliding segments are called sliding cycles in PWS systems.

At the same time, there has in recent years been an attempt, for example by  J. Llibre and co-workers, to determine the maximum number of crossing limit cycles in PWL systems. Certain subcases have been solved \cite{esteban2021a,li2021a,llibre2013a} and more generally it has been shown in \cite{carmona2023a} that the number of crossing limit cycles is bounded by $8$. To the best of our knowledge, only $3$ crossing limit cycles have been realized (see \cite{Llibre3LC,HuanYang}). We also refer to \cite{Gasull2020,LlibreOrd,Freire,BragaMello,Han2010} and references therein.

 The purpose of this paper is to begin the analysis of sliding cycles of regularized PWL systems. In this paper, we focus on the $VI_3$ case, see Fig. \ref{fig:2examples}(a), and show (when the slow divergence integral is not identically zero) that the family of canard cycles $\cup_{x\in J}\Gamma_x$ (blue in Fig. \ref{fig:2examples}(a)), with $J\subset \mathbb R_+$ being a compact interval, can produce at most $2$ limit cycles of \eqref{Zreg}. This will follow from \cite{RHKK} and Theorem \ref{theorem-linearzeros} in Section \ref{section-mainresult}, which states that the slow divergence integral has at most $1$ zero counting multiplicity in $J$ (Remark \ref{remark-cycl} in Section \ref{section-mainresult}). The fold point from ``below'' in terms of $Z^-$ is invisible and the graphic $\Gamma_x$ consists of the orbit of $Z^-$ from $(x,0)$ to $(\Pi(x),0)$ and the segment $[\Pi(x),x]\subset \{y=0\}$, see Fig. \ref{fig:2examples}(a). $\Gamma_x$ is a sliding cycle of the PWS Filippov system, but following \cite{RHKK} we call it a canard cycle because it contains both stable and unstable sliding portions of the discontinuity line $y=0$ . We also prove the existence of $2$ sliding limit cycles for some regularized PWL $VI_3$ systems (Section \ref{section-mainresult} and Section \ref{section-proof}). 
 \smallskip
 
 In a separate paper \cite{RHKKProgress}, we consider the remaining cases, including the $VV_1$ case (Fig. \ref{fig:2examples}(b)). More precisely, sliding limit cycles can be produced by canard cycles detected in the half-plane with visible fold point (see red graphics in Fig. \ref{fig:2examples}). We expect the existence of $3$ sliding cycles in regularized PWL $VI_3$ and $VV_1$ systems.  
 
 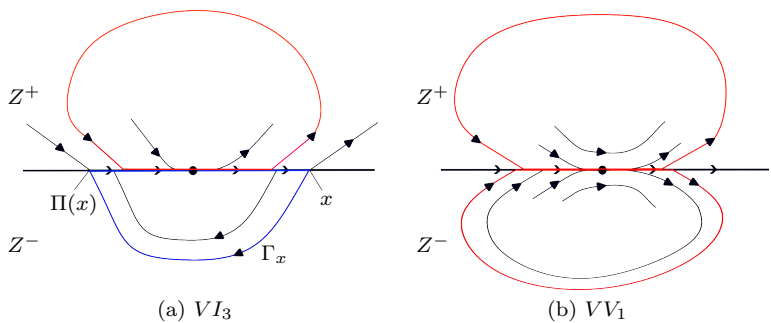
\begin{figure}[htb]
	\begin{center}
		\includegraphics[width=9.8cm,height=3.7cm]{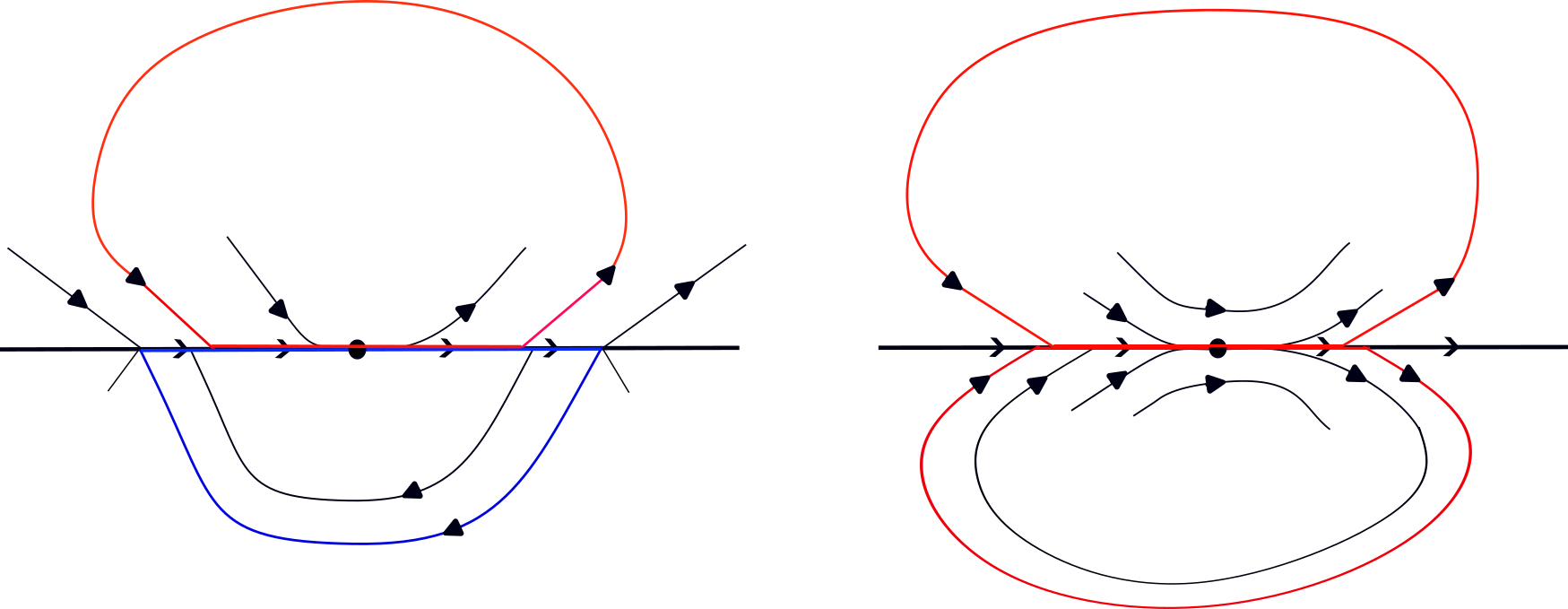}
		{\footnotesize
\put(-229,-10){(a) $VI_3$}
\put(-83,-10){(b) $VV_1$}
\put(-168,32){$x$}
\put(-269,30){$\Pi(x)$}
\put(-190,11){$\Gamma_x$}
\put(-285,70){$Z^+$}
\put(-285,14){$Z^-$}
\put(-132,70){$Z^+$}
\put(-132,14){$Z^-$}
  }
         \end{center}
	\caption{Canard cycles through two-fold singularities with sliding (the $VI_3$ case and the $VV_1$ case, see \cite{Kuznetsov2003}). In this paper we study the number of limit cycles near canard cycles $\Gamma_x$ (blue) located in the half-plane $y\le 0$ with invisible fold point, in the $VI_3$ case. $\Pi$ is the Poincar\'{e} half-map defined in Section \ref{subsection-halfmap}.  Canard cycles (red) can also appear in the half-plane with visible fold point. }
	\label{fig:2examples}
\end{figure}

 The paper is organized as follows:
 In Section \ref{section-applications}, we first review some basic concepts of Filippov PWS systems and present a normal form for the PWL $VI_3$-case, see Section \ref{section-pws} and Section \ref{section-twofold}. In Section \ref{section-regularized}, we define our regularized PWL $VI_3$ two-fold model and introduce the notion of slow divergence integral. Finally, in Section
 \ref{subsection-halfmap} we present some results on a Poincar\'e half-map based on the work of \cite{Carmona}. Subsequently, we then state our main result Theorem \ref{theorem-linearzeros} in Section \ref{section-mainresult}. The proof of  Theorem \ref{theorem-linearzeros}, available in Section \ref{section-proof}, uses the characterization of the Poincar\'e half-map presented in Section \ref{subsection-halfmap} and deals with the different cases (saddle, focus, proper and improper node, etc.) separately. 
 \section{Background and statement of the main result}
\label{section-applications}

\subsection{Filippov PWS systems}\label{section-pws}
In the following, we review the most basic concepts of PWS systems. For this purpose, we will follow \cite[Section 2]{RHKK}. Notice that henceforth we write $Z^\pm (h) := \nabla h \cdot Z^\pm$ for the Lie-derivative of $h$ in the direction $Z^\pm$.  

The discontinuity set $\Sigma=\{h(z)=0\}$ of \eqref{pws} is frequently called the switching manifold \cite{Bernardo08,GST2011} and it is divided into three disjoint sets $\Sigma=\Sigma^{cr}\cup \Sigma^{sl}\cup \Sigma^T$ characterized in the following way:
\begin{itemize}
 \item[(1)] The subset $\Sigma^{cr}\subset \Sigma$ consisting of all points $q\in \Sigma$ where
 \begin{align*}
  Z^+(h)(q)Z^-(h)(q)>0,
 \end{align*}
 is called ``crossing'', see Fig. \ref{fig:pws} (orange).
\item[(2)] The subset $\Sigma^{sl}\subset \Sigma$ consisting of all points $q\in \Sigma$ where
\begin{align*}
  Z^+(h)(q)Z^-(h)(q)<0
 \end{align*}
 is called ``sliding''. It is said to be stable (resp. unstable) if $Z^+(h)(q)<0$ and $Z^-(h)(q)>0$ (resp. $Z^+(h)(q)>0$ and $Z^-(h)(q)<0$). Fig. \ref{fig:pws} illustrates (in pink) stable sliding (unstable sliding can be obtained by reversing the arrows). 
 \item[(3)] The subset $\Sigma^T\subset \Sigma$ consisting of all points $q\in \Sigma$ where either $Z^+(h)(q)=0$ or $Z^-(h)(q)=0$ is called the set of tangency points. If $q\in \Sigma^T$ and $Z^+(h)(q)=0$ then $q$ is called a tangency point from above. Tangency points from below are defined similarly. Finally, $q$ is a double tangency point if it is tangency point from above and from below.
 \end{itemize}
 
  \begin{figure}[htb]
	\begin{center}
		\includegraphics[width=.5\textwidth]{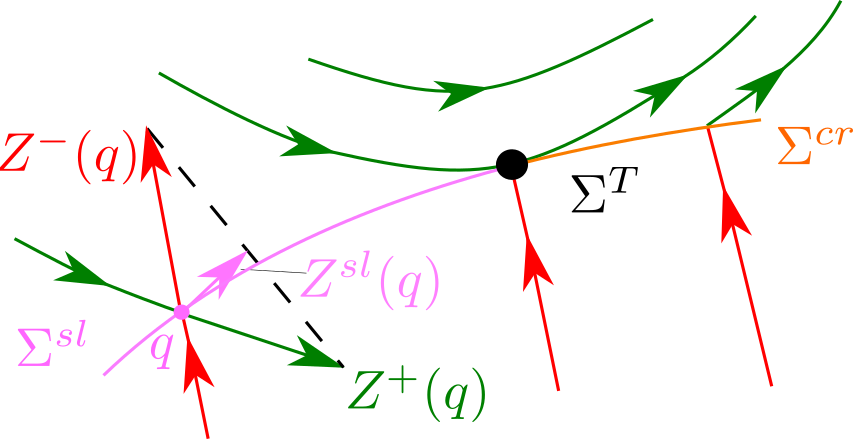}
		         \end{center}
	\caption{Illustration of a PWS visible fold $\Sigma^T$, that locally divides $\Sigma$ into a crossing set $\Sigma^{cr}$ (orange) and a sliding set $\Sigma^{sl}$ (pink, stable in the present case). On $\Sigma^{sl}$ we assign Filippov's sliding vector-field $Z^{sl}$, which is defined as the convex combination of $Z^\pm(q) $ that is tangent to $\Sigma$ at $q\in \Sigma$. }
	\label{fig:pws}
\end{figure}
 
 Along $\Sigma^{cr}$, trajectories can be extended from $Z^+$ to $Z^-$ or from $Z^-$ to $Z^+$ by concatenating orbits of $Z^+$ and $Z^-$. In contrast, trajectories of $Z^\pm$ reach $\Sigma^{sl}$ in finite time and to be able to continue trajectories a vector-field has to be defined along $\Sigma^{sl}$. The most common way to do this is by using the Filippov convention, where the \textit{sliding vector-field} $Z^{sl}$ is assigned on $\Sigma^{sl}$:
 \begin{align}
  Z^{sl}(q) = Z^+(q) p(q)+Z^-(q)(1-p(q)), \quad p(q): =\frac{-Z^-(h)}{Z^+(h)-Z^-(h)}(q),\label{Zsl}
 \end{align}
 for $q\in \Sigma^{sl}$,
see Fig. \ref{pws}. In this way, one can define a forward solution and a backward solution through any point, see \cite{filippov1988differential,Kuznetsov2003}. These solutions are (clearly) not unique in general, but this allows us to define $\omega$ and $\alpha$-limit sets. The choice \eqref{Zsl} is motivated by examples \cite{Bernardo08}, but importantly it also relates to the $\epsilon\rightarrow 0$ limit of \eqref{Zreg}, see \cite{RHKK} and Section \ref{section-regularized} below.

Frequently, we will suppose that $h(x,y)=y$, which is without loss of generality in the PWL case (and locally in the nonlinear case). In this case, $Z^{sl}=(X^{sl},0)$ which defines $X^{sl}$ (a one-dimensional vector-field on $\Sigma^{sl}\subset\{y=0\}$).

 We further classify the points in $\Sigma^T$ as follows (see also \cite{Bernardo08}):
 \begin{itemize}
  \item[(4)] A point $q\in \Sigma^T$ is a fold point from ``above'' if the orbit of $Z^+(\cdot)$ through $q$ has a quadratic tangency with $\Sigma$ at $q$, i.e.
  \begin{align*}
  \begin{cases}
  Z^+(q)&\ne 0,\\ 
  Z^+(h)(q)&=0, \\
  (Z^+)^2(h)(q)&\ne 0.\end{cases}
  \end{align*} We define a fold point from ``below'' in terms of $Z^-$ in a similar way. 
  \item[(5)] A fold point $q\in \Sigma^T$ from ``above'' is said to be visible, if the orbit of $Z^+(\cdot)$ through $q$ is contained within $y>0$ in neighborhood of $q$. It is said to be invisible otherwise. In terms of Lie-derivatives, we clearly have $(Z^+)^2(h)(q)>0$ iff $q$ satisfying $Z^+(q)\ne 0$ and $Z^+(h)(q)=0$ is visible. Fold points from below are classified in a similar way. In particular, $(Z^-)^2(h)(z)<0$ iff $q$ satisfying $Z^-(q)\ne 0$ and $Z^-(h)(q)=0$ is visible.
 \end{itemize}
 The fold point illustrated in Fig. \ref{fig:pws} (black dot) is visible from above.

%
\subsection{Two-folds and a normal form for the PWL $VI_3$-case}\label{section-twofold}
 Now, we finally arrive at the concept of two-folds in PWS systems. These are double tangency points that are fold points for both vector-fields and play the role of canard points in the analysis of (\ref{Zreg}), see \cite{RHKK}.
 \begin{itemize}
 \item[(6)] A two-fold $q\in \Sigma^T$ is a point with quadratic tangencies from above \textit{and} from below. In terms of Lie-derivatives we have: 
  \begin{align*}
  \begin{cases}
  Z^\pm (q)&\ne 0,\\ 
  Z^\pm (h)(q)&=0, \\
  (Z^\pm) ^2(h)(q)&\ne 0,\end{cases}
  \end{align*}
  with these equations understood to hold for \textit{both} $\pm$.
 \item[(7)] A two-fold is said to be visible-visible, visible-invisible, invisible-invisible according to the ``visibility'' of the fold from above and below, respectively, see item (5) above. 
\end{itemize}


The paper \cite{Kuznetsov2003} gave a characterization of two-folds. There are seven different cases, two cases of visible-visible (called $VV_{1,2}$), three cases of visible-invisible ($VI_{1,2,3}$), and finally two cases of invisible-invisible ($II_{1,2}$). The different subcases (of visible-visible, visible-invisible and invisible-invisible) are determined by (a) whether there is sliding ($VV_1$, $VI_2$, $VI_3$ and $II_1$) or not ($VV_2$, $VI_1$ and $II_2$), and (if there is sliding:) by (b) the direction of sliding flow and finally (c) whether the unfolding leads to singularities on $\Sigma^{sl}$ of the sliding vector-field, see also \cite[Fig. 2]{kristiansen2015a}. In the present paper,  we focus on the $VI_3$-case which we characterize in the following result.
  \begin{proposition}\proplab{VI3}
Consider a PWS system with a $VI_3$ two-fold. Then there exist local coordinates such that
\begin{equation}\label{VI3cond1}
\begin{aligned}
 \begin{cases}
  X^+(0,0) &>0,\\
  Y^+(0,0) &=0,\\
  \frac{\partial}{\partial x} Y^+(0,0) &>0,
 \end{cases}\quad 
 \begin{cases}
  X^-(0,0) &<0,\\
  Y^-(0,0) &=0,\\
  \frac{\partial}{\partial x} Y^-(0,0) &<0,
 \end{cases}
\end{aligned}
\end{equation}
and
\begin{align}
\left(X^-\frac{\partial}{\partial x} Y^+-X^+\frac{\partial}{\partial x} Y^-\right)(0,0)>0.\label{VI3cond2}
\end{align}
In particular, in the PWL case, then there exists an invertible affine map $\Phi:(x,y)\mapsto (\tilde x,\tilde y)$ such that $\widetilde Z^\pm:=\Phi^*(Z^\pm)$ satisfy
\begin{align}\label{linearPWS}
 \widetilde Z^- (\tilde x,\tilde y)&=\begin{bmatrix}
-1                    +  \beta_- \tilde y\\
                          - \tilde x+\gamma_- \tilde y\end{bmatrix},\quad
                     \widetilde Z^+(\tilde x,\tilde y) = 
                     \begin{bmatrix}
                    B+\alpha_+ \tilde x + \beta_+ \tilde y\\
                    \delta_+\tilde x + \gamma_+ \tilde y\end{bmatrix}.
                    \end{align}
                    Here $B>\delta_+>0$, $\Sigma^{sl} = ]-\infty,0[\cup ]0,\infty[$, $\Sigma^T =\{0\}$,
                    \begin{align*}
                    \gamma_- = \operatorname{tr}(D Z^-),\quad \beta_-= \operatorname{det}(D Z^-),
                    \end{align*}
                    and
                    \begin{align*}
                        \widetilde X^{sl}(x) = \frac{1}{1+\delta_+} \left(B-\delta_++\alpha_+x\right).
                    \end{align*}
\end{proposition}
\begin{proof}
The first part of the proposition follows from the definition of the $VI_3$-case, see e.g. \cite[Section 2.2]{RHKK}. Now regarding the PWL case, we write
\begin{align*}
Z^\pm(z) = A^\pm z + b^\pm,
\end{align*}
with $A^\pm = (A_{ij}^\pm)$ and $b^\pm=(b^\pm_1,b^\pm_2)^T$. Then
the conditions \eqref{VI3cond1} and \eqref{VI3cond2} become
\begin{align*}
 \begin{cases}
  b^+_1 &>0,\\
  b^+_2 &=0,\\
  A_{21}^+ &>0,
 \end{cases}\quad 
 \begin{cases}
  b^-_1 &<0,\\
  b^-_2 &=0,\\
  A_{21}^- &<0,
 \end{cases}
\end{align*}
and
\begin{align}
b_1^-A_{21}^+-b_1^+A_{21}^->0.\label{VI3cond2b}
\end{align}
 Since $A_{21}^-\ne 0$, we can transform $Z^-$ into a Li\'{e}nard form $\widetilde Z^-$ with $\widetilde A_{11}^-=0$ using an $y$-fibered transformation defined by 
 \begin{align*}
  (x,y)\mapsto \tilde x = x -\frac{A_{11}^-}{A_{21}^-}y.
 \end{align*}
Dropping the tildes, we can then subsequently apply the scalings
\begin{align*}
 (x,y)\mapsto \begin{cases} \tilde x &= \vert b_1^-\vert^{-1} x,\\
               \tilde y &= \vert b_1^-\vert^{-1} \vert A_{21}^-\vert^{-1}   y
              \end{cases}
              \end{align*}
This gives a $\widetilde{Z}^-$ with 
\begin{align*}
\widetilde A^- = \begin{bmatrix} 
0 & \operatorname{det}(A^-)\\
-1 & \operatorname{tr}(A^-)
\end{bmatrix},\quad 
\tilde b^-=\begin{bmatrix}-1\\0 \end{bmatrix}.
\end{align*}
Setting $\delta_+:=\widetilde A_{21}^+,\, B:=\tilde b_1^+$, \eqref{VI3cond2b} becomes $B-\delta_+>0$. The expression for $\widetilde X^{sl}$ follows easily from \eqref{Zsl}. This completes the proof.
\end{proof}

In the rest of this paper, when we refer to \eqref{linearPWS}, we use $ Z^\pm, X^{sl}$ instead of $\widetilde Z^\pm,\widetilde X^{sl}$.


\subsection{Regularized PWL $VI_3$ two-fold and the slow divergence integral}\label{section-regularized}
We now consider (\ref{Zreg}) with $h(x,y)=y$ in the form
\begin{align}\label{regul-normalnovo}
  \dot z &=Z^+(z,\lambda)\phi(y\epsilon^{-2}) +Z^-(z,\lambda) (1-\phi(y\epsilon^{-2})),
 \end{align}
where $0<\epsilon\ll 1$, $\lambda\sim 0\in \mathbb R$, $Z^\pm(\cdot,\lambda)=\left(X^\pm(\cdot,\lambda),Y^\pm(\cdot,\lambda)\right)$ are planar affine vector-fields, depending smoothly on a parameter $\lambda$, and $Z^\pm(\cdot,0)=Z^\pm(\cdot)$, with $Z^\pm(\cdot)$ defined in \eqref{linearPWS}. 
We add the following technical assumptions on $\phi$.
\begin{enumerate}[label=({A}{{\arabic*}})]
%
 \item \label{assA} The function $\phi$ has the following asymptotics when $s\to\pm\infty$: \begin{align*}
 \phi(s)\rightarrow \begin{cases}
                     1 & \text{for}\quad  s\rightarrow \infty,\\
                     0 & \text{for}\quad s\rightarrow -\infty.
                    \end{cases}
\end{align*}
\item \label{assB} The function $\phi$ is strictly monotone, i.e.,  $\phi'(s)>0$ for all $s\in \mathbb R$.
\item \label{assC} The function $\phi$ is smooth at $\pm \infty$ in the following sense: Each of the functions
\begin{align*}
 \phi_+(s):=\begin{cases}
             1 & \text{for}\quad s=0,\\
             \phi(s^{-1}) & \text{for}\quad s>0,
            \end{cases},\quad
            \phi_-(s):=\begin{cases}
             \phi(-s^{-1}) & \text{for}\quad s>0,\\
             0 & \text{for}\quad s=0,
            \end{cases}
\end{align*}
are smooth at $s=0$. 
%
\end{enumerate}
%
Assumption \ref{assC} means that (\ref{regul-normalnovo}) is a regular perturbation of $Z^+(\cdot,\lambda)$ or $Z^-(\cdot,\lambda)$ outside any fixed neighborhood of $y=0$, see \cite{RHKK}. 
Moreover, it is well-known (see \cite{Sotomayor96} and \cite[Theorem 2.2]{RHKK}) that once Assumption \ref{assB} holds, sliding in \eqref{linearPWS} implies existence of local invariant manifolds for \eqref{regul-normalnovo}, which carry a reduced flow that is a regular perturbation of $\dot x=X^{sl}(x)$, with $X^{sl}$ given in \propref{VI3}:
\begin{equation}
    \label{sliding-linear}
    X^{sl}(x)=\frac{1}{1+\delta_+}\left(B-\delta_++\alpha_+ x\right).
\end{equation}
When $\alpha_+\ne 0$, $X^{sl}$ has a simple zero 
\begin{equation}
    \label{simple-zero}
    x^*=-\frac{B-\delta_+}{\alpha_+}\ne 0,
\end{equation}
and when $\alpha_+=0$, $X^{sl}(x)$ is positive for all $x\in \mathbb R$.

\smallskip

The following assumption plays an important role when we study the existence and number of sliding limit cycles of \eqref{regul-normalnovo}, see \cite{RHKK}.

\begin{enumerate}[label=({A}{4})]
    \item \label{assExtra} We assume that 
\begin{equation}
    \label{NZvelocity}
    \frac{\partial}{\partial\lambda} Y^- \frac{\partial}{\partial x} Y^+\ne \frac{\partial}{\partial\lambda} Y^+ \frac{\partial}{\partial x} Y^-
\end{equation}
at $(z,\lambda)=(0,0)$.
\end{enumerate}
Let us explain the meaning of Assumption \ref{assExtra}. The fold point $(x,y)=(0,0)$ from above and below is persistent to smooth perturbations of \eqref{linearPWS}. Indeed, The Implicit Function Theorem and \eqref{VI3cond1} imply the existence of smooth $\lambda$-families of fold points $z_+=(x_+(\lambda),0)$ from above in terms of $Z^+(\cdot,\lambda)$ and fold points $z_-=(x_-(\lambda),0)$ from below in terms of $Z^-(\cdot,\lambda)$, for $\lambda\sim 0$, with $x_\pm(0)=0$. If we assume non-zero velocity of the collision between $z_+$ and $z_-$ for $\lambda=0$ at the origin $z=0$:
$$ x_+'(0)-x_-'(0)=\left(-\frac{\frac{\partial}{\partial\lambda} Y^+}{\frac{\partial}{\partial x} Y^+}+\frac{\frac{\partial}{\partial\lambda} Y^-}{\frac{\partial}{\partial x} Y^-}\right)(0,0)\ne 0,$$
then we get \eqref{NZvelocity}. 

\smallskip

Following \cite[Section 3]{RHKK}, to study the existence and number of sliding limit cycles of \eqref{regul-normalnovo} produced by the canard cycle $\Gamma_x$ (Fig. \ref{fig:2examples}(a)) for $(\epsilon,\lambda)\sim (0,0)$, we use the slow divergence integral associated to the segment $[\Pi(x),x]$ at level $\lambda=0$:
\begin{equation}\label{slowdivnovo}
 I(x) =\int_{\Pi(x)}^x \frac{(Y^+-Y^-)^2}{X^-Y^+-X^+Y^-}(u,0,0) \phi'\left(\phi^{-1} \left(\frac{-Y^-}{Y^+-Y^-}(u,0,0)\right)\right)du,
\end{equation}
for $x>0$. See (3.1) in \cite{RHKK}. Now, if we use \eqref{linearPWS}, then \eqref{slowdivnovo} becomes 

\begin{equation}
    \label{SDI-linearPWS}
    I(x)=(1+\delta_+)\phi'\left(\phi^{-1}\left(\frac{1}{1+\delta_+}\right)\right)\int_{\Pi(x)}^{x}\frac{udu}{X^{sl}(u)}.
\end{equation}
Notice that $I(0)=0$ and that the expression in front of the integral in \eqref{SDI-linearPWS} is positive. The domain of $I$ is a subset of the domain of $\Pi$ (Section \ref{subsection-halfmap}) and it depends on the location of $x^*$ defined in \eqref{simple-zero}. More precisely, the domain of $I$ is the biggest subset $[0,b[$ ($b>0$ or $b=+\infty$) of the domain of $\Pi$ such that the sliding vector field $X^{sl}$ given in \eqref{sliding-linear} is positive on $[\Pi(x),x]$, for all $x\in [0,b[$.  It should be clear that the domain of $I$ is equal to the domain of $\Pi$ when $\alpha_+=0$. For more details see later sections.
\smallskip

\begin{remark}
    As emphasized by \eqref{SDI-linearPWS}, in the PWL case the $\phi$-dependent term of the integrand of $I$ in \eqref{slowdivnovo} is a constant and can therefore go outside the integration. In this sense, our analysis in the PWL case will be independent of the regularization function. This is in contrast to the general case, see \cite{RHKK}. Here we constructed an arbitrary number of limit cycles by varying $\phi$, even in the case where $Z^-$ is affine and $Z^+$ is quadratic.
\end{remark}
We assume that
\[\lambda=\epsilon\widetilde{\lambda}\]
where $\widetilde{\lambda}\sim 0$. We denote by $\cycl(\Gamma_{x_0})$ the cyclicity of the canard cycle $\Gamma_{x_0}$ inside \eqref{regul-normalnovo}, for $x_0\in ]0,b[$. More precisely, we say that the cyclicity of $\Gamma_{x_0}$ inside \eqref{regul-normalnovo} is bounded by $N\in\mathbb N$ if there exist $\epsilon_0>0$, $\delta_0>0$ and a neighborhood $\mathcal U$ of $0$ in the $\widetilde{\lambda}$-space such that \eqref{regul-normalnovo} has at most $N$
limit cycles, lying within Hausdorff distance $\delta_0$
of $\Gamma_{x_0}$, for all $(\epsilon,\widetilde{\lambda})\in ]0,\epsilon_0]\times \mathcal U$. We call the
smallest $N$ with this property the cyclicity of $\Gamma_{x_0}$ and denote it by $\cycl(\Gamma_{x_0})$. 

 We define $\cycl(\cup_{x\in J}\Gamma_x)$ in a similar way, where $J=[\theta,b-\theta]$ (resp. $J=[\theta,\frac{1}{\theta}]$)  for $b>0$ (resp. $b=+\infty$), with any small and fixed $\theta>0$. 
 
 The following theorem is a direct consequence of \cite{RHKK}.
 \begin{theorem}
     \label{theorem-cyclicity-RHKK} Consider \eqref{regul-normalnovo} and suppose that Assumptions \ref{assA} through \ref{assExtra} are satisfied.
     Then the following statements are true.
     \begin{enumerate}
         \item If $I(x_0)<0$ (resp. $I(x_0)> 0$), then $\cycl(\Gamma_{x_0})=1$ and the limit cycle is hyperbolic and attracting (resp. repelling) when it exists. Moreover, if $I$ has no zeros in $]0,b[$, then $\cycl(\cup_{x\in J}\Gamma_x)=1$. 
         \item If $I$ has a zero of multiplicity $l\ge 1$ at $x=x_0$, then $\cycl(\Gamma_{x_0})\le l+1$. When $I$ has at most $l\ge 1$ zeros in $]0,b[$, counting multiplicity, then we have $\cycl(\cup_{x\in J}\Gamma_x)\le l+1$.
         \item Suppose that $I$ has exactly $l\ge 1$ simple zeros $x_1<\dots<x_{l}$ in $]0,b[$. If $x_{l+1}\in ]x_{l},b[$, then there is a smooth function $\widetilde{\lambda}=\widetilde{\lambda}_c(\epsilon)$, with $\widetilde{\lambda}_c(0)=0$, such that \eqref{regul-normalnovo} with $Z^\pm(\cdot,\epsilon \widetilde{\lambda}_c(\epsilon))$ has $l+1$ periodic orbits $\mathcal{O}_1^\epsilon,\dots,\mathcal{O}_{l+1}^\epsilon$, for each $\epsilon\sim 0$ and $\epsilon>0$. The periodic orbit $\mathcal{O}_i^\epsilon$ is isolated, hyperbolic and Hausdorff close to the canard cycle $\Gamma_{x_i}$, for each $i=1,\dots, l+1$.
     \end{enumerate}
 \end{theorem}
 \begin{proof}
     Theorem \ref{theorem-cyclicity-RHKK}.1 and Theorem \ref{theorem-cyclicity-RHKK}.2 follow from \cite[Proposition 3.2]{RHKK}, and Theorem \ref{theorem-cyclicity-RHKK}.3 follows from \cite[Theorem 3.1]{RHKK}.
 \end{proof}

\subsection{Poincar\'{e} half-map}\label{subsection-halfmap}

In this section, we focus on $Z^-$ defined in \eqref{linearPWS} (remember we drop the tildes). 
The study of the transition map (often called Poincar\'{e} half-map) from $x>0,y=0$ to $x<0,y=0$ by following the orbits of $Z^-$ in forward time  can be found in \cite{Carmona}. We denote by $\Pi$ the Poincar\'{e} half-map  (Fig. \ref{fig:2examples}(a)).  
From \cite[Theorem 8]{Carmona} and \cite[Theorem 19]{Carmona} it follows that we can use an integral characterization for the Poincar\'{e} half-map $\Pi$:

\begin{equation}
    \label{intcharact}
    \int_{\Pi(x)}^x\frac{-udu}{V(u)}=0,
\end{equation}
where
\begin{align*}
    V(x) = \beta_-x^2-\gamma_-x+1.
\end{align*}
Notice that $V$ is related to the characteristic polynomial associated with $Z^-$:
\begin{align*}
    P(\lambda) = \lambda^2-\gamma_- \lambda + \beta_-,
\end{align*}
by
\begin{align*}
    V(u) = u^2 P(u^{-1}),
\end{align*}
for $u\ne 0$. From this, it can be easily seen that the following lemma holds.
\begin{lemma}
    \label{lemma-LRkappa}
    The following statements are true.
    \begin{enumerate}
        \item If $\beta_-\ne 0$, then $Z^-$ defined in \eqref{linearPWS} has a singularity at $(x,y)=(\frac{\gamma_-}{\beta_-},\frac{1}{\beta_-})$ with eigenvalues 
\begin{equation}\label{eigenvaluesZ-}
    \varkappa_\pm=\frac{\gamma_-\pm\sqrt{\gamma_-^2-4\beta_-}}{2}.
\end{equation}
\item If $\beta_-\ne 0$ and $\gamma_-^2-4\beta_-\ge 0$, then the invariant affine eigenline corresponding to the eigenvalue $\varkappa_\pm$ in \eqref{eigenvaluesZ-} is given by 
\begin{equation}\label{eigenlinene0}
x=\varkappa_\mp y+\frac{1}{\varkappa_\mp},
\end{equation}
and it intersects the $x-$axis at the zero $x=\frac{1}{\varkappa_\mp}$ of the polynomial $V$. 
\item If $\beta_-=0$ and $\gamma_-\ne 0$, then the line 
\begin{equation}\label{eigenline=0}
x=\gamma_-y+\frac{1}{\gamma_-}
\end{equation}
is invariant w.r.t. $Z^-$. Moreover, the line intersects the $x-$axis at the zero $x=\frac{1}{\gamma_-}$ of $V$.
    \end{enumerate}
\end{lemma}

We denote by $x_L$ and $x_R$ the zeros of $V$ in Lemma \ref{lemma-LRkappa}.2: 
\begin{equation}\label{intersect-pointsRL}
x_L,\ x_R=\frac{1}{\varkappa_\pm},
\end{equation}
where we assume that $x_L<x_R$ if $x_L\ne x_R$.  
\smallskip

The set $\Gamma_x\cup\text{Int}(\Gamma_x)$ belongs to the class ($S_0$) defined in \cite{Carmona} (that is, $\Gamma_x\cup\text{Int}(\Gamma_x)$ contains no singularities of $Z^-$). For more details we refer to Sections \ref{saddle-subsection}--\ref{focus-subsection}, Appendix \ref{subsection-centerappen} and Appendix \ref{nosingularities-subsection}. The Poincar\'{e} half-map $\Pi$ can be extended to $\Pi(0)=0$ and it is analytic in its domain of definition (see \cite{Carmona}). The following discussion is based on Lemma \ref{lemma-LRkappa}. If $\gamma_-^2-4\beta_-<0$, then $Z^-$ has a focus or center in $\{y>0\}$, and the domain of $\Pi$ is $[0,+\infty[$ and the image of $\Pi$ is $]-\infty,0]$ (Fig. \ref{fig:foc-cent} in Section \ref{focus-subsection} and Fig. \ref{fig:centonly} in Appendix \ref{subsection-centerappen}). When $\beta_-<0$, $Z^-$ has a hyperbolic saddle in $\{y<0\}$ and the stable (resp. unstable) straight manifold of the saddle intersects the $x$-axis at $x=x_R>0$ (resp. $x=x_L<0$). In this case the domain of $\Pi$ is $[0,x_R[$ and the image of $\Pi$ is $]x_L,0]$ (Fig. \ref{fig:sedlo} and Section \ref{saddle-subsection}). When $Z^-$ has a hyperbolic node in $\{y>0\}$ with distinct eigenvalues ($\beta_->0$ and $\gamma_-^2-4\beta_->0$), then two straight-line solutions corresponding to the eigenvalues intersect the $x$-axis at $x=x_L$ and $x=x_R$ with $0<x_L<x_R$ or $x_L<x_R<0$. If $0<x_L<x_R$ (resp. $x_L<x_R<0$), then the domain of $\Pi$ is $[0,x_L[$ (resp. $[0,+\infty[$) and the image of $\Pi$ is $]-\infty,0]$ (resp. $]x_R,0]$). We refer to Fig. \ref{fig:cvor} and Section \ref{node-subsection}. System $Z^-$ may have a hyperbolic node in $\{y>0\}$ with repeated eigenvalues ($\beta_->0$ and $\gamma_-^2-4\beta_-=0$). In this case we have one straight-line solution (corresponding to the eigenvalue) which intersects the $x$-axis at $x=x_R\ne 0$. If $x_R>0$ (resp. $x_R<0$), then the domain of $\Pi$ is $[0,x_R[$ (resp. $[0,+\infty[$) and the image of $\Pi$ is $]-\infty,0]$ (resp. $]x_R,0]$). If $Z^-$ has no singularities ($\beta_-=0$), then there exists an invariant line intersecting the $x$-axis at $x=\frac{1}{\gamma_-}$ ($\gamma_-\ne 0$) or orbits of $Z^-$ are parabolas $y=\frac{1}{2}x^2+c$ ($\gamma_-= 0$). In the former case, the domain and image of $\Pi$ are respectively $[0,\frac{1}{\gamma_-}[$ and $]-\infty,0]$ if $\gamma_->0$ or $[0,+\infty[$ and $]\frac{1}{\gamma_-},0]$ if $\gamma_-<0$,
whereas in the latter case the domain and image of $\Pi$ are respectively $[0,+\infty[$ and $]-\infty,0]$. We refer to Fig. \ref{fig:nosing} and Appendix \ref{nosingularities-subsection}.
\smallskip

The function $V$ is positive on the domain and image of $\Pi$. Using \eqref{intcharact} we get 
\begin{equation}
    \label{Pi-derivative}
    \Pi'(x)=\frac{xV(\Pi(x))}{\Pi(x)V(x)}.
\end{equation}
We have $\Pi'<0$ and $\Pi'(0)=-1$. 

\subsection{Main result}\label{section-mainresult}

Recall that the slow divergence integral $I$ is given by \eqref{SDI-linearPWS}.
Our goal is to study the number of zeros (counting multiplicity) of $I$ in $]0,b[$. We show that $I$ is either identically zero or has at most $1$ zero (counting multiplicity) in $]0,b[$.
\begin{theorem}
    \label{theorem-linearzeros}
    If ($\beta_-\ne 0$ and $(\alpha_+,\gamma_-)\ne (0,0)$) or ($\beta_-= 0$ and $\alpha_++\gamma_-(B-\delta_+)\ne 0$), then $I$ has at most $1$ zero counting multiplicity in $]0,b[$. Moreover, there exist parameter values in the $(\beta_-,\gamma_-,B,\alpha_+,\delta_+)$-space for which this condition is satisfied and $I$ has a simple zero in $]0,b[$. 
    
    If the condition is not satisfied, then $I$ is identically zero. 
\end{theorem}

\begin{remark}\label{remark-cycl}
If the condition of Theorem \ref{theorem-linearzeros} is satisfied, then $\cycl(\cup_{x\in J}\Gamma_x)\le 2$. This follows directly from Theorem \ref{theorem-cyclicity-RHKK}.2 and Theorem \ref{theorem-linearzeros}. Since the slow divergence integral $I$ can have a simple zero in $]0,b[$ by Theorem \ref{theorem-linearzeros}, a direct consequence of Theorem \ref{theorem-cyclicity-RHKK}.3 is that there exists a regularized PWL system \eqref{regul-normalnovo} with $2$ hyperbolic limit cycles. We also refer to Theorem \ref{theorem-saddle}, Theorem \ref{thm-distinctnode}, Theorem \ref{thm-repeatedeigen} and Theorem \ref{thm-focuscase}.
\end{remark}

 \begin{remark}
      We point out that limit cycles of \eqref{regul-normalnovo} near so-called boundary graphics $\Gamma_0$ (the origin $(x,y)=(0,0)$) and $\Gamma_b$ cannot be studied using Theorem \ref{theorem-linearzeros} and Theorem \ref{theorem-cyclicity-RHKK}. The graphic $\Gamma_b$ can contain (1) the zero $x^*$ of the sliding vector-field $X^{sl}$ as its corner point (in this case $b=x^*$ if $x^*>0$ or $b=\Pi^{-1}(x^*)$ if $x^*<0$), see e.g. Fig. \ref{fig:sedlo}(b) and (d), (2) a hyperbolic saddle located away from the line $y=0$, see Fig. \ref{fig:sedlo}(a), (c) and (e), (3) both the corner point $x^*$ and the hyperbolic saddle away from $y=0$, etc.
      This -- along with the description of canard cycles in the half-plane with a visible fold -- are topics of further study.
\end{remark}

In this paper we do  not treat periodic orbits of \eqref{regul-normalnovo} when $I\equiv 0$. We leave this to future work. We expect that the analysis of this case will depend upon the regularization function. 

\section{Proof of Theorem \ref{theorem-linearzeros}}\label{section-proof}

Let us denote by $\tilde{I}(x)$ the integral in \eqref{SDI-linearPWS}. Using \eqref{Pi-derivative} it follows that 
\begin{equation}
  \label{SDI-derivative-linear}  
 \tilde{I}'(x)=x(x-\Pi(x))\frac{\alpha_+\beta_-x\Pi(x)+\beta_-(B-\delta_+)(x+\Pi(x))-\alpha_+-\gamma_-(B-\delta_+)}{(1+\delta_+)X^{sl}(x)X^{sl}(\Pi(x))V(x)},
\end{equation}
for all $ x\in [0,b[$. From \eqref{SDI-derivative-linear}, $\Pi(x)<0$ for $x>0$ and the definition of the domain of $\Pi$ and $I$ it follows that 
\begin{equation}
    \label{equivfunctions}
    \tilde{I}'(x)=\Psi(x)\Delta(x), 
\end{equation}
where $\Psi(x)>0$ for all $x\in ]0,b[$ and
\begin{equation}
    \label{Delta-equiv}
    \Delta(x)=\alpha_+\beta_-x\Pi(x)+\beta_-(B-\delta_+)(x+\Pi(x))-\alpha_+-\gamma_-(B-\delta_+).\nonumber
\end{equation}
Using \eqref{equivfunctions} it is clear that $x=x_0\in ]0,b[$ is a zero of multiplicity $l$ of $\tilde{I}'$ if and only if $x=x_0\in ]0,b[$ is a zero of multiplicity $l$ of $\Delta$.
\smallskip

Recall the condition given in Theorem \ref{theorem-linearzeros}:
\begin{equation}
    \label{condition-theorem-imp}
    (\beta_-\ne 0 \text{ and } (\alpha_+,\gamma_-)\ne (0,0)) \text{ or } (\beta_-= 0 \text{ and } \alpha_++\gamma_-(B-\delta_+)\ne 0).
\end{equation}
Suppose first that the condition \eqref{condition-theorem-imp} is not satisfied. Thus, we assume that $(\alpha_+,\gamma_-)=(0,0)$ or ($\beta_-= 0$ and $\alpha_++\gamma_-(B-\delta_+)= 0$). Then $I\equiv 0$. Indeed, when $(\alpha_+,\gamma_-)=(0,0)$, then $\Pi(x)=-x $ because the linear system $Z^-$ is invariant under the symmetry $(x,t)\to(-x,-t)$ for $\gamma_-=0$. This and the fact that the integrand in \eqref{SDI-linearPWS} is an odd function ($\alpha_+=0$) imply that $I$ is identically zero. When $\beta_-= 0$ and $\alpha_++\gamma_-(B-\delta_+)= 0$, then from \eqref{SDI-derivative-linear} it follows that $I'\equiv 0$ and, since $I(0)=0$, we have that $I$ is identically zero.  
\smallskip

In the rest of this section we suppose that the condition \eqref{condition-theorem-imp} is satisfied. We may assume that $\gamma_-\ge 0$. Indeed, system \eqref{linearPWS} is invariant under the symmetry  $(x,\alpha_+,\gamma_+,\gamma_-,t)\to (-x,-\alpha_+,-\gamma_+,-\gamma_-,-t)$, and, if we denote by $\Pi_{\beta_-,\gamma_-}$ (resp. $I_{\beta_-,\gamma_-,B,\alpha_+,\delta_+}$) the Poincar\'{e} half-map $\Pi$ (resp. the slow divergence integral $I$) of \eqref{linearPWS}, then using \eqref{SDI-linearPWS} we get
$$I_{\beta_-,\gamma_-,B,\alpha_+,\delta_+}(x)=-I_{\beta_-,-\gamma_-,B,-\alpha_+,\delta_+}(-\Pi_{\beta_-,\gamma_-}(x)).$$
$I_{\beta_-,-\gamma_-,B,-\alpha_+,\delta_+}$ is the slow divergence integral of \eqref{linearPWS} where $\alpha_+,\gamma_+,\gamma_-$ are replaced with $-\alpha_+,-\gamma_+,-\gamma_-$.
Since $\Pi_{\beta_-,\gamma_-}'<0$, the above formula implies that $x>0$ is a zero of $I_{\beta_-,\gamma_-,B,\alpha_+,\delta_+}$ if and only if $-\Pi_{\beta_-,\gamma_-}(x)>0$ is a zero of $I_{\beta_-,-\gamma_-,B,-\alpha_+,\delta_+}$ (with the same multiplicity). We conclude that the case $\gamma_-<0$ follows from the case $\gamma_->0$.

First we prove the following theorem.
\begin{theorem}
    \label{thm-mainpart1}
    Consider the slow divergence integral $I(x)$, with $ x\in [0,b[$, defined in \eqref{SDI-linearPWS}. The following statements are true.
    \begin{enumerate}
        \item Suppose that $\beta_-= 0$ and $\alpha_++\gamma_-(B-\delta_+)\ne 0$. Then the interval $[0,b[$ is bounded and, if $\alpha_++\gamma_-(B-\delta_+)>0$ (resp. $\alpha_++\gamma_-(B-\delta_+)<0$), then $ I<0$ (resp. $ I>0$) on $]0,b[$. For any small $\theta>0$, we have $\cycl(\cup_{x\in[\theta,b-\theta]}\Gamma_x)=1$. The limit cycle is attracting (resp. repelling) if it exists.
        \item Suppose that $\alpha_+\beta_-\ne 0$ and $\gamma_-=0$. Then the interval $[0,b[$ is bounded and, if $\alpha_+>0$ (resp. $\alpha_+<0$), then $ I<0$ (resp. $ I>0$) on $]0,b[$. For any small $\theta>0$, $\cycl(\cup_{x\in[\theta,b-\theta]}\Gamma_x)=1$ and the limit cycle is attracting (resp. repelling) if it exists.   
    \end{enumerate}
\end{theorem}
\begin{proof}
  \textit{Statement 1.} Suppose that $\beta_-= 0$ and $\alpha_++\gamma_-(B-\delta_+)\ne 0$. If $\gamma_-=0$, then $\alpha_+\ne 0$ and the sliding vector field $X^{sl}$ has a simple zero at $x=x^*$ defined in \eqref{simple-zero}. This and the fact that for $\beta_-=\gamma_-=0$ the domain and image of $\Pi$ are respectively $[0,+\infty[$ and $]-\infty,0]$ (Section \ref{subsection-halfmap}) imply that the domain $[0,b[$ of $I$ is bounded. If $\gamma_->0$, then the domain of $\Pi$ is $[0,\frac{1}{\gamma_-}[$ (Section \ref{subsection-halfmap}), and $[0,b[\subset [0,\frac{1}{\gamma_-}[$ is bounded.

  Since $\beta_-= 0$, we have
  $$\Delta(x)=-\alpha_+-\gamma_-(B-\delta_+),$$
 with $\Delta$ defined in \eqref{equivfunctions}. Now, if $\alpha_++\gamma_-(B-\delta_+)>0$ (resp. $\alpha_++\gamma_-(B-\delta_+)<0$), then $\tilde{I}'(x),\Delta(x)<0$ (resp. $\tilde{I}'(x),\Delta(x)>0$) for all $x\in ]0,b[$. Since $\tilde{I}(0)=0$, we have that $\tilde{I}$ and $I$ are negative (resp. positive) on $]0,b[$. The rest of the statement follows directly from Theorem \ref{theorem-cyclicity-RHKK}.1.\\
 \\
 \textit{Statement 2.} Suppose that $\alpha_+\beta_-\ne 0$ and $\gamma_-=0$. If $\beta_-<0$, then $[0,b[$ is bounded because the domain of $\Pi$ is bounded (see the saddle case in Section \ref{subsection-halfmap}). If $\beta_->0$, then the domain and image of $\Pi$ are respectively $[0,+\infty[$ and $]-\infty,0]$ (see the center case in Section \ref{subsection-halfmap}). Since $x^*$ is well-defined ($\alpha_+\ne 0$), this implies that $[0,b[$ is bounded.

 Since $\gamma_-=0$, we have $\Pi(x)=-x $ and 
 $$\Delta(x)=-\alpha_+ V(x).$$
 Notice that $V(x)=\beta_-x^2+1$ when $\gamma_-=0$. The function $V$ is positive on the domain of $\Pi$. If $\alpha_+>0$ (resp. $\alpha_+<0$), then $\tilde{I}'(x)<0$ (resp. $\tilde{I}'(x)>0$) for all $x\in ]0,b[$. The rest of the proof is now analogous to the proof in Statement $1$.
\end{proof}
For more details about the domain $[0,b[$ of $I$ in Theorem \ref{thm-mainpart1} see Appendix \ref{subsection-centerappen} and Appendix \ref{nosingularities-subsection}.

It remains to study the case where $\beta_-\ne 0$ and $\gamma_->0$ (Section \ref{subsection-important}--Section \ref{focus-subsection}). Theorem \ref{theorem-saddle} in Section \ref{saddle-subsection} (the saddle case), Theorem \ref{thm-distinctnode} and Theorem \ref{thm-repeatedeigen} in Section \ref{node-subsection} (the node case) and Theorem \ref{thm-focuscase} in Section \ref{focus-subsection} (the focus case) imply that for $\beta_-\ne 0$ and $\gamma_->0$ $I$ has at most $1$ zero counting multiplicity in $]0,b[$. Moreover, there exist parameter values in the $(\beta_-,\gamma_-,B,\alpha_+,\delta_+)$-space, with $\alpha_+<0$, $\beta_-\ne 0$ and $\gamma_->0$, for which $I$ has a simple zero in $]0,b[$. 
\smallskip

The main idea in the proof of the above mentioned theorems is to show that, for $\beta_-\ne 0$ and $\gamma_->0$, there is at most $1$ intersection (multiplicity counted)
between the curve $y=\Pi(x)$ and the curve $\overline \Delta(x,y)=0$ in the fourth quadrant with $x\in]0,b[$ where
\begin{equation}
    \label{bar Delta-equiv}
    \overline \Delta(x,y)=\alpha_+\beta_-xy+\beta_-(B-\delta_+)(x+y)-\alpha_+-\gamma_-(B-\delta_+).
\end{equation}
Then this implies that $\Delta(x)=\overline \Delta(x,\Pi(x))$ (or, equivalently, $\tilde{I}'$) has at most $1$ zero counting multiplicity in $]0,b[$. Using
Rolle’s theorem and $\tilde{I}(0)=0$, we conclude that $\tilde{I}$ (or $I$) has at most $1$ zero counting multiplicity in $]0,b[$. For a similar idea see e.g. \cite{HuzakPrey}.
\smallskip

In Section \ref{subsection-important} we classify the curves defined by the equation $\overline \Delta(x,y)=0$ and find contact points between these curves and the orbits of system \eqref{3degreepolysys} defined in Section \ref{subsection-important}.

\subsection{Properties of the curves defined by $\overline \Delta=0$}\label{subsection-important}

Consider the function $\overline \Delta$ defined in \eqref{bar Delta-equiv}. In further details, notice first that  
\begin{align}
\overline \Delta(x,y)=\overline \Delta(y,x),\label{symmetryDelta}\\
    \nabla \overline \Delta(x,y) = \begin{bmatrix}
(1+\delta_+) \beta_- X^{sl}(y)\\
(1+\delta_+) \beta_- X^{sl}(x)
    \end{bmatrix},\label{gradDelta}\\
    \overline \Delta(x,x^*) = -\alpha_+V(x^*).\label{barDeltaprop}
\end{align}
Then, for $\beta_-\ne 0$ and $\gamma_->0$, we distinguish between the following $3$ cases.\\
\\
1. If $\alpha_+\beta_-\ne 0$, $\gamma_->0$ and $V(x^*)=0$, with $x^*$ defined in \eqref{simple-zero}, then by  \eqref{symmetryDelta} and \eqref{barDeltaprop}
\begin{equation}\label{unionlines}
    \overline \Delta(x,y)=\alpha_+\beta_-(x-x^*)(y-x^*),
    \end{equation}
and $\overline \Delta(x,y)=0$ therefore represents the union of two lines $x=x^*$ and $y=x^*$. In this case we will see that $y=\Pi(x)$ and $\overline \Delta(x,y)=0$ have no intersection
points. For further details, see Section \ref{saddle-subsection} and Section \ref{node-subsection}.\\
\\
2. If $\alpha_+\beta_-\ne 0$, $\gamma_->0$ and $V(x^*)\ne 0$, then $\overline \Delta(x,y)=0$ represents a hyperbola
\begin{equation}
\label{hyperbola}
y=hp(x):=
\frac{\alpha_++\gamma_-(B-\delta_+)-\beta_-(B-\delta_+)x}{\beta_-(1+\delta_+)X^{sl}(x)}.
\end{equation}
It follows from \eqref{symmetryDelta} that the function $hp$ is an involution. The graph of $y=hp(x)$ has a vertical asymptote $x=x^*$ and a horizontal asymptote $y=x^*$, and 
\begin{equation}\label{deriv-hyper}
hp(0)=\frac{\gamma_-x^*-1}{\beta_-x^*}, \  
\ hp'(x)=-\frac{\alpha_+^2V(x^*)}{\beta_-(1+\delta_+)^2X^{sl}(x)^2}.
\end{equation}
In this case we will prove that $y=\Pi(x)$ and $\overline \Delta(x,y)=0$ have at most $1$ intersection counting multiplicity in the fourth quadrant with $x\in ]0,b[$. Moreover, we show the existence of a transversal intersection for some parameter values satisfying the above condition. See Sections \ref{saddle-subsection}--\ref{focus-subsection}.\\ 
\\
3. If $\alpha_+=0$, $\beta_-\ne 0$ and $\gamma_-> 0$, then $\overline \Delta(x,y)=0$ represents a line
\begin{equation}
    \label{line-equation}
    y=\frac{\gamma_-}{\beta_-}-x.
\end{equation}
In this case we prove that $y=\Pi(x)$ and $\overline \Delta(x,y)=0$ have no intersection
points. See Sections \ref{saddle-subsection}--\ref{focus-subsection}.

It can be easily seen that in cases $1$--$3$ we have $\overline \Delta(x_L,x_R)=\overline \Delta(x_R,x_L)=0$ where $x_L$ and $x_R$ are defined in \eqref{intersect-pointsRL}.

\begin{lemma}
    \label{lemma-below-above}
    Suppose that $\beta_-\ne 0$ and $\gamma_->0$. Consider the function $\Delta(x)$ defined in \eqref{equivfunctions}, with $x\in ]0,b¨[$, and $\overline \Delta$ defined in  \eqref{bar Delta-equiv}. The following statements are true.
    \begin{enumerate}
    \item Suppose that $\beta_-<0$. If the curve $\overline \Delta(x,y)=0$ lies above (resp. below) the curve $y=\Pi(x)$ for all $x$ kept in an interval $J\subset]0,b[$, then $\Delta(x),\tilde I'(x)>0$ (resp. $\Delta(x),\tilde I'(x)<0$) for all $x\in J$.
    \item Suppose that $\beta_->0$. If the curve $\overline \Delta(x,y)=0$ lies above (resp. below) the curve $y=\Pi(x)$ for all $x$ kept in an interval $J\subset]0,b[$, then $\Delta(x),\tilde I'(x)<0$ (resp. $\Delta(x),\tilde I'(x)>0$) for all $x\in J$.
    \end{enumerate}
\end{lemma}
\begin{proof} We will prove Statement $1$. Statement $2$ can be proved in the same fashion as Statement $1$. 

Let $\beta_-<0$.
 Suppose first that $\alpha_+\ne 0$ and $V(x^*)=0$. Then $\overline \Delta$ is given in \eqref{unionlines}. If the curve $\overline \Delta(x,y)=0$ (that is, $y=x^*$) lies above the curve $y=\Pi(x)$ with $x\in J\subset]0,b[$, then $x^*>\Pi(x)$ for all $x\in J$. Using \eqref{unionlines} we get
$$\Delta(x)=\overline \Delta(x,\Pi(x))=\alpha_+\beta_-(x-x^*)(\Pi(x)-x^*)>0, \ \forall x\in J.$$
We used $\beta_-<0$, $x^*>\Pi(x)$ for all $x\in J$ and $\alpha_+(x-x^*)>0$ for all $x\in J$. Let us prove that $\alpha_+(x-x^*)>0$ for all $x\in J$. If $\alpha_+>0$, then \eqref{simple-zero} implies that $x^*<0$. Since $x>0$ for all $x\in J$, it follows that $\alpha_+(x-x^*)>0$ for all $x\in J$. If $\alpha_+< 0$, then $x^*>0$. Using the definition of the domain $[0,b[$ of $I$ (Section \ref{section-mainresult}), it is clear that $b\le x^*$. Thus, $x<x^*$ for all $x\in J$. This implies that $\alpha_+(x-x^*)>0$ for all $x\in J$. The case where the curve $\overline \Delta(x,y)=0$ (that is, $y=x^*$) lies below the curve $y=\Pi(x)$ can be studied in a similar way. We get $\Delta(x)<0$
for all $x\in J$.

 Now, suppose that $\alpha_+\ne 0$ and $V(x^*)\ne 0$. If the curve $\overline \Delta(x,y)=0$ lies above the curve $y=\Pi(x)$ with $x\in J$, then $hp(x)>\Pi(x)$ for all $x\in J$, with $hp$ defined in \eqref{hyperbola}. If we substitute \eqref{hyperbola} in $hp(x)>\Pi(x)$ and if we use $\beta_-<0$, then we get $\Delta(x)>0$ for all $x\in J$. The case where the curve $\overline \Delta(x,y)=0$ lies below the curve $y=\Pi(x)$ can be studied in a similar way. We obtain $\Delta(x)<0$
for all $x\in J$.

 Finally, suppose that $\alpha_+=0$. If the curve $\overline \Delta(x,y)=0$ lies above the curve $y=\Pi(x)$ with $x\in J$, then $\frac{\gamma_-}{\beta_-}-x>\Pi(x)$ for all $x\in J$. We use \eqref{line-equation}. Since $\beta_-<0$, we have $\Delta(x)>0$ for all $x\in J$. The case where the curve $\overline \Delta(x,y)=0$ lies below the curve $y=\Pi(x)$ can be studied in a similar way. We have $\Delta(x)<0$
for all $x\in J$.
\end{proof}

\smallskip

Notice that $y=\Pi(x)$ is the $x\ge 0$-subset of the stable manifold of
the hyperbolic saddle point $(x,y)=(0,0)$ of the following polynomial system of degree $3$
\begin{equation}\label{3degreepolysys}
\begin{aligned}
 \dot x &=yV(x),\\
 \dot y &=xV(y).
\end{aligned}
 \end{equation}
This can be easily seen from \eqref{Pi-derivative} (see also \cite[Remark 16]{Carmona}). It is clear that system \eqref{3degreepolysys} is invariant under the symmetry $(x,y)\to (y,x)$. 
It is important (Sections \ref{saddle-subsection}--\ref{focus-subsection}) to calculate the number of contact points between the orbits of system \eqref{3degreepolysys} and the curve $\overline \Delta(x,y)=0$. The contact points are solutions of  
\begin{equation}\label{system-contactpoints}
\begin{aligned}
 \nabla\overline \Delta(x,y)\cdot\left(yV(x),xV(y)\right)  &=0,\\
 \overline \Delta(x,y)&=0.
\end{aligned}
 \end{equation}
Using \eqref{gradDelta} the first equation in \eqref{system-contactpoints} becomes 
\begin{equation}
    \label{contact-gradient}
    (1+\delta_+)\beta_-\left(xV(y)X^{sl}(x)+yV(x)X^{sl}(y)\right)=0
\end{equation}
Recall that $X^{sl}(x^*)=0$ for $\alpha_+\ne 0$. Therefore if $\alpha_+\beta_-\ne 0$, $\gamma_->0$ and $V(x^*)=0$, then \eqref{contact-gradient} implies that all points on the lines $x=x^*$ and $y=x^*$ are the contact points.
On the other hand, if $\alpha_+\beta_-\ne 0$, $\gamma_->0$ and $V(x^*)\ne 0$ and if we substitute \eqref{hyperbola} in \eqref{contact-gradient}, we get the following equation for  contact points
\begin{equation}
    \label{contactpointsCase1}
   V(x)\left(\alpha_++\gamma_-(B-\delta_+)+\alpha_+\beta_-x^2\right)=0. 
\end{equation}
Using \eqref{contactpointsCase1} the contact points are:
$(x,y)=(x_L,x_R)$, $(x,y)=(x_R,x_L)$ (if $\gamma_-^2-4\beta_-\ge 0$), and $(x,y)=(x_C,-x_C)$ and $(x,y)=(-x_C,x_C)$ 
where 
\begin{equation}
\label{contact-important}
x_C=\sqrt{\frac{\gamma_-x^*-1}{\beta_-}},
\end{equation}
if $\frac{\gamma_-x^*-1}{\beta_-}\ge 0$. Let us recall that $x_L,x_R$ are defined in \eqref{intersect-pointsRL} and $x^*$ is defined in \eqref{simple-zero}.

When $\beta_-\ne 0$, $\gamma_->0$ and $\alpha_+=0$, then we substitute \eqref{line-equation} in \eqref{contact-gradient} and get the following equation for contact points:
\begin{equation}
    \label{contactpointsCase2}
   V(x)=0. 
\end{equation}

\subsection{The saddle case}\label{saddle-subsection}
In this section we suppose that $\beta_-<0$. Then $Z^-$ has a hyperbolic saddle at $(x,y)=(\frac{\gamma_-}{\beta_-},\frac{1}{\beta_-})$ with eigenvalues $\varkappa_-<0$ and $\varkappa_+>0$ given in \eqref{eigenvaluesZ-}. From \eqref{eigenlinene0} it follows that the stable manifold of the hyperbolic saddle is given by $x=\varkappa_+y+x_R$ and the unstable manifold is given by $x=\varkappa_-y+x_L$
where $x_L<0$ and $x_R>0$ are defined in \eqref{intersect-pointsRL}.
It is clear that the stable (resp. unstable) manifold intersects the $x$-axis at $x=x_R$ (resp. $x=x_L$). We refer to Fig. \ref{fig:sedlo}.
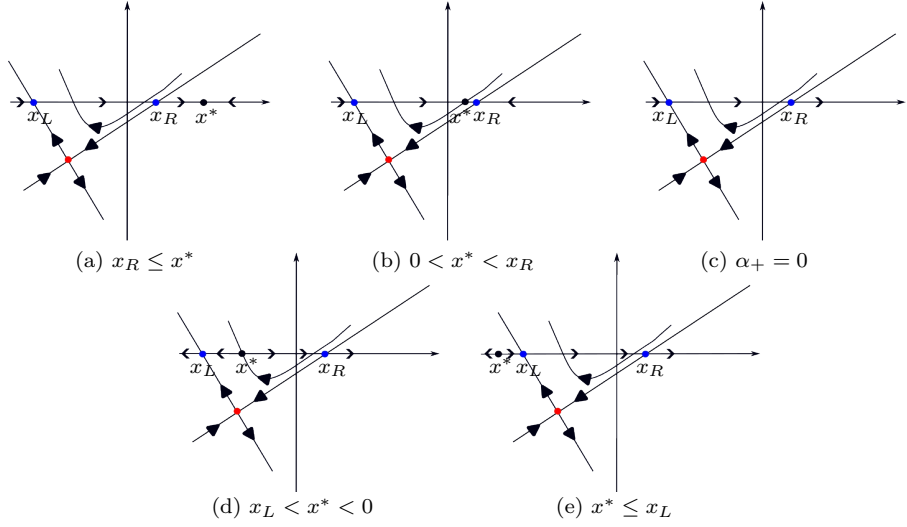
\begin{figure}[htb]
	\begin{center}
		\includegraphics[width=11.8cm,height=6.5cm]{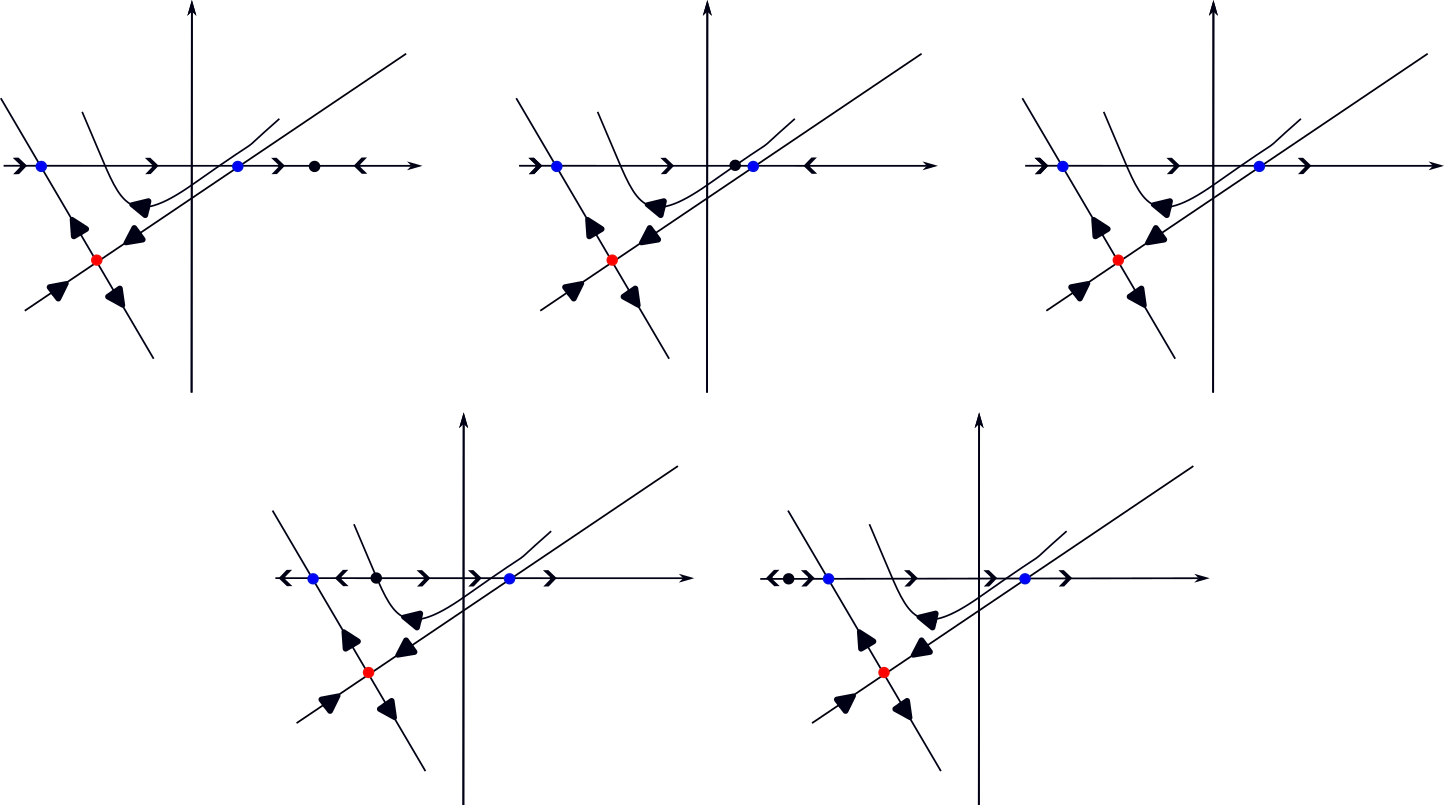}
		{\footnotesize
  \put(-283,140){\footnotesize{$x_R$}}
  \put(-329,140){\footnotesize{$x_L$}}
  \put(-266,139){\footnotesize{$x^*$}}
  \put(-162,140){\footnotesize{$x_R$}}
  \put(-211,140){\footnotesize{$x_L$}}
  \put(-171,139){\footnotesize{$x^*$}}
  \put(-47,140){\footnotesize{$x_R$}}
  \put(-95,140){\footnotesize{$x_L$}}
\put(-220,45){\footnotesize{$x_R$}}
  \put(-268,45){\footnotesize{$x_L$}}
  \put(-251,45){\footnotesize{$x^*$}}
\put(-101,45){\footnotesize{$x_R$}}
  \put(-146,45){\footnotesize{$x_L$}}
  \put(-156,45){\footnotesize{$x^*$}}
   \put(-311,85){\footnotesize{(a) $x_R\le x^*$}}
   \put(-200,85){\footnotesize{(b) $0<x^*< x_R$}}
   \put(-77,85){\footnotesize{(c) $\alpha_+=0$}}
   \put(-260,-8){\footnotesize{(d) $x_L<x^*<0$}}
   \put(-131,-8){\footnotesize{(e) $x^*\le x_L$}}
  
}
         \end{center}
	\caption{Phase portraits of $Z^-$, with $\beta_-<0$ and $\gamma_->0$, defined in \eqref{linearPWS} and the direction of the sliding vector field \eqref{sliding-linear} along $y=0$. $Z^-$ has a hyperbolic saddle. We do not draw the corresponding phase portraits of $Z^+$.}
	\label{fig:sedlo}
\end{figure}
  
  From Fig. \ref{fig:sedlo} it follows that the domain and image of $\Pi$ are respectively $[0,x_R[$ and $]x_L,0]$ (see also \cite{Carmona}). The domain of the slow divergence integral $I$ (or $\tilde I$) depends on the location of the singularity $x=x^*$ of the sliding vector field. We distinguish between $5$ cases.
\begin{enumerate}
    \item[(a)] If $\alpha_+<0$ (hence $x^*>0$) and $x_R\le x^*$, then the domain of $I$ is $[0,x_R[$ and we consider the canard cycle $\Gamma_x$ for all $x\in ]0,x_R[$ (see Fig. \ref{fig:sedlo}(a)).
    \item[(b)] If $\alpha_+<0$  and $x^*< x_R$, then the domain of $I$ is $[0,x^*[$ and we consider the canard cycle $\Gamma_x$ for all $x\in ]0,x^*[$ (see Fig. \ref{fig:sedlo}(b)).
\item[(c)] If $\alpha_+=0$, then we have the same domain of $I$ as in the case (a) (see Fig. \ref{fig:sedlo}(c)).
 \item[(d)] If $\alpha_+>0$ (hence $x^*<0$) and $x_L<x^*$, then the domain of $I$ is $[0,\Pi^{-1}(x^*)[$ and we consider the canard cycle $\Gamma_x$ for all $x\in ]0,\Pi^{-1}(x^*)[$ (see Fig. \ref{fig:sedlo}(d)).
 \item[(e)] If $\alpha_+>0$ (hence $x^*<0$) and $x^*\le x_L$, then we deal with the same domain of $I$ as in the case (a) (see Fig. \ref{fig:sedlo}(e)).
\end{enumerate}

Besides the hyperbolic saddle at the origin, the system \eqref{3degreepolysys} has hyperbolic and attracting nodes at $(x,y)=(x_R,x_R)$ and $(x,y)=(x_L,x_L)$, and hyperbolic and repelling nodes at $(x,y)=(x_R,x_L)$ and $(x,y)=(x_L,x_R)$. Notice that the lines $x=x_R$, $x=x_L$, $y=x_R$ and $y=x_L$ are invariant (Fig. \ref{fig:contactsaddle}). Let us focus on the singularity $(x,y)=(x_R,x_L)$ in the fourth quadrant. Since $\gamma_->0$, it is easy to see that the straight-line solution corresponding to the weaker eigenvalue of $(x,y)=(x_R,x_L)$ is $x=x_R$, and the regular orbit of \eqref{3degreepolysys} given by $y=\Pi(x)$ tends to $(x,y)=(x_R,x_L)$ tangentially to the straight-line $x=x_R$ (in backward time).

\begin{figure}[htb]
	\begin{center}
		\includegraphics[width=12.5cm,height=9.5cm]{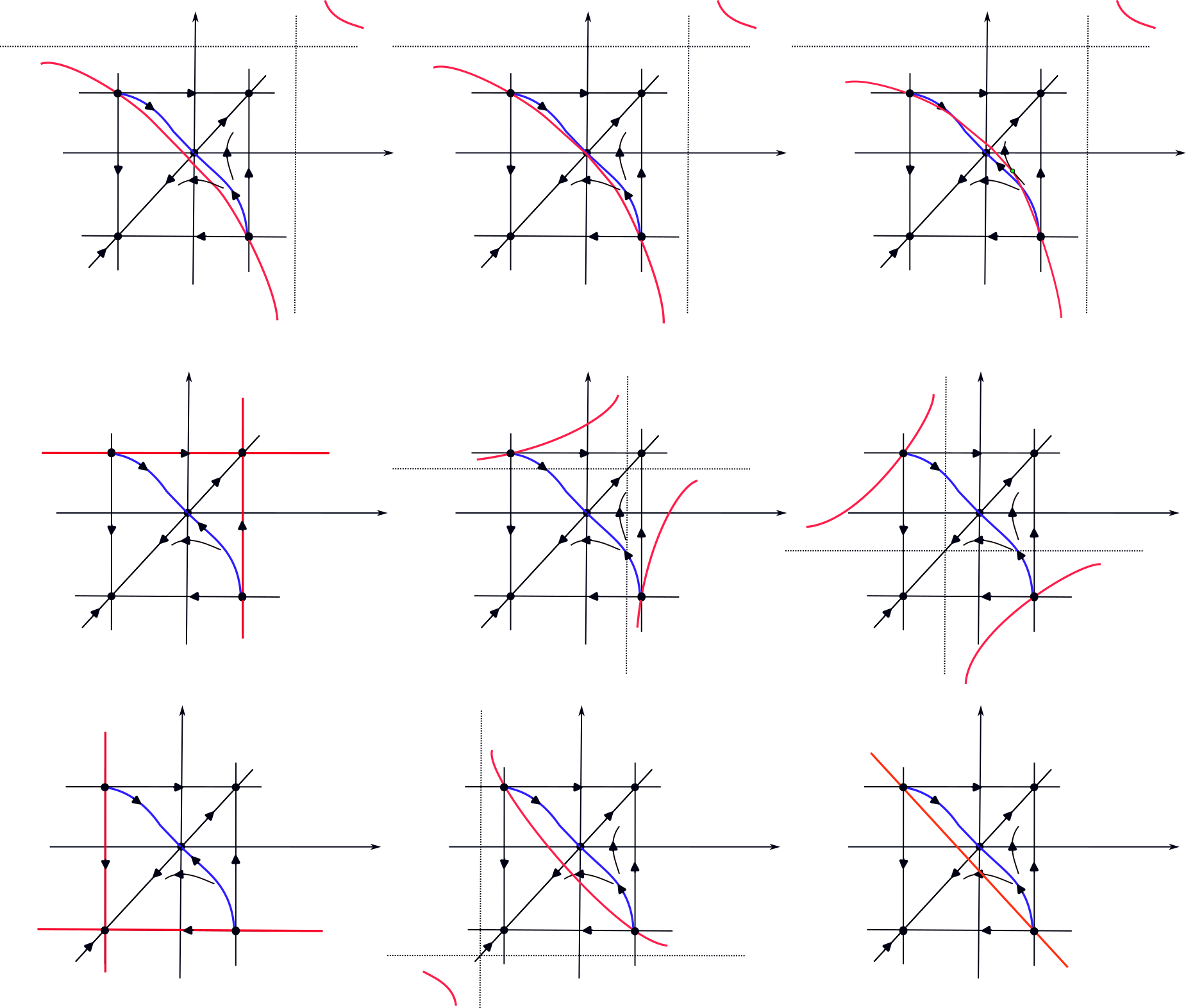}
		{\footnotesize
   \put(-280,224){\footnotesize{$x_R$}}
   \put(-332,224){\footnotesize{$x_L$}}
   \put(-320,176){\footnotesize{(a) $\frac{1}{\gamma_-}<x^*$}}
   \put(-161,224){\footnotesize{$x_R$}}
   \put(-215,224){\footnotesize{$x_L$}}
   \put(-203,176){\footnotesize{(b) $x^*=\frac{1}{\gamma_-}$}}
   \put(-42,224){\footnotesize{$x_R$}}
   \put(-94,224){\footnotesize{$x_L$}}
    \put(-100,176){\footnotesize{(c) $x_R<x^*<\frac{1}{\gamma_-}$}}
    \put(-281,127){\footnotesize{$x_R$}}
   \put(-333,127){\footnotesize{$x_L$}}
    \put(-320,84){\footnotesize{(d) $x^*=x_R$}}
   \put(-161,127){\footnotesize{$x_R$}}
   \put(-215,127){\footnotesize{$x_L$}}
    \put(-210,84){\footnotesize{(e) $0<x^*<x_R$}}
   \put(-42,127){\footnotesize{$x_R$}}
   \put(-96,127){\footnotesize{$x_L$}}
    \put(-100,82){\footnotesize{(f) $x_L<x^*<0$}}
   \put(-283,37){\footnotesize{$x_R$}}
   \put(-335,37){\footnotesize{$x_L$}}
    \put(-320,-2){\footnotesize{(g) $x^*=x_L$}}
   \put(-163,37){\footnotesize{$x_R$}}
   \put(-203,37){\footnotesize{$x_L$}}
   \put(-200,-2){\footnotesize{(h) $x^*<x_L$}}
    \put(-43,37){\footnotesize{$x_R$}}
   \put(-96,37){\footnotesize{$x_L$}}
   \put(-85,-2){\footnotesize{(i) $\alpha_+=0$}}
}
         \end{center}
	\caption{The phase portrait of \eqref{3degreepolysys} for $\beta_-<0$ and $\gamma_->0$, with the curve $\overline \Delta(x,y)=0$ (red). The part of the blue curve located in the fourth quadrant is the graph of $\Pi$. We draw the vertical and horizontal lines $x=x^*$ and $y=x^*$ using dashed lines. $(x,y)=(x_C,-x_C)$ and $(x,y)=(-x_C,x_C)$ are the intersection points between the red curve and $y=-x$. We indicate the contact point $(x,y)=(x_C,-x_C)$ when $x_C$ is positive and contained in the domain of the slow divergence integral $I$ (Fig \ref{fig:contactsaddle}(c)).}
	\label{fig:contactsaddle}
\end{figure}
A detailed statement of Theorem \ref{theorem-saddle} below covers all possible mutual positions of the curve $y=\Pi(x)$ and the curve $\overline \Delta(x,y)=0$ (see Fig. \ref{fig:contactsaddle}).
\begin{theorem}\label{theorem-saddle}
Suppose that $\beta_-<0$ and $\gamma_->0$. Then $x_R<\frac{1}{\gamma_-}$ and the following statements are true.
\begin{enumerate}
    \item{($\alpha_+<0$)} If $\frac{1}{\gamma_-}<x^*$ (Fig. \ref{fig:contactsaddle}(a)), then we have $ I<0$ on $]0,x_R[$ and, for any small $\theta>0$, $\cycl(\cup_{x\in[\theta,x_R-\theta]}\Gamma_x)=1$. The limit cycle is attracting.
   \item{($\alpha_+<0$)} If $x^*=\frac{1}{\gamma_-}$ (Fig. \ref{fig:contactsaddle}(b)), then we have $I<0$ on $]0,x_R[$ and, for any small $\theta>0$, $\cycl(\cup_{x\in[\theta,x_R-\theta]}\Gamma_x)=1$. The limit cycle is attracting. 
    \item{($\alpha_+<0$)} If $x_R<x^*<\frac{1}{\gamma_-}$ (Fig. \ref{fig:contactsaddle}(c)), then the function $I$ has at most $1$ zero (counting multiplicity) on $]0,x_R[$ and, for any small $\theta>0$, $\cycl(\cup_{x\in[\theta,x_R-\theta]}\Gamma_x)\le 2$. There exists $x^*$ (sufficiently close to $\frac{1}{\gamma_-}$) such that $I$ has a simple zero in $]0,x_R[$ and then, for any sufficiently small $\theta>0$, $\cycl(\cup_{x\in[\theta,x_R-\theta]}\Gamma_x)=2$.  
    \item{($\alpha_+<0$)} If $x^*=x_R$ (Fig. \ref{fig:contactsaddle}(d)), then $I>0$ on $]0,x_R[$ and, for any small $\theta>0$, $\cycl(\cup_{x\in[\theta,x_R-\theta]}\Gamma_x)=1$ (the limit cycle is repelling).
     \item{($\alpha_+<0$)} If $0<x^*<x_R$ (Fig. \ref{fig:contactsaddle}(e)), then we have $I>0$ on $]0,x^*[$ and, for any small $\theta>0$, $\cycl(\cup_{x\in[\theta,x^*-\theta]}\Gamma_x)=1$ (the limit cycle is repelling).
      \item{($\alpha_+>0$)} If $x_L<x^*<0$ (Fig. \ref{fig:contactsaddle}(f)), then we have $I<0$ on $]0,\Pi^{-1}(x^*)[$ and, for any small $\theta>0$, $\cycl(\cup_{x\in[\theta,\Pi^{-1}(x^*)-\theta]}\Gamma_x)=1$ (the limit cycle is attracting).
      \item{($\alpha_+>0$)} If $x^*=x_L$ (Fig. \ref{fig:contactsaddle}(g)), we have $ I<0$ on $]0,x_R[$ and, for any small $\theta>0$,  $\cycl(\cup_{x\in[\theta,x_R-\theta]}\Gamma_x)=1$ (the limit cycle is attracting).
        \item{($\alpha_+>0$)} If $x^*<x_L$ (Fig. \ref{fig:contactsaddle}(h)), then we have $ I<0$ on $]0,x_R[$ and, for any small $\theta>0$, $\cycl(\cup_{x\in[\theta,x_R-\theta]}\Gamma_x)=1$ (the limit cycle is attracting).
        \item If $\alpha_+=0$ (Fig. \ref{fig:contactsaddle}(i)), we have $I<0$ on $]0,x_R[$ and, for any small $\theta>0$, $\cycl(\cup_{x\in[\theta,x_R-\theta]}\Gamma_x)=1$ (the limit cycle is attracting).
     
\end{enumerate}
\end{theorem}

\begin{proof} Suppose that $\beta_-<0$ and $\gamma_->0$. Using \eqref{intersect-pointsRL} we know that
$$x_L=\frac{2}{\gamma_--\sqrt{\gamma_-^2-4 \beta_-}}<0, \ \ x_R=\frac{2}{\gamma_-+\sqrt{\gamma_-^2-4 \beta_-}}>0,$$
$V(x_L)=V(x_R)=0$ where $V(x)=\beta_-x^2-\gamma_-x+1$. The graph of $V$ is concave down because $\beta_-<0$.
Using the expression for $x_R$ it can be easily seen that $x_R<\frac{1}{\gamma_-}$. Notice that $x_C\ge 0$ in \eqref{contact-important} is well-defined if $x^*\le\frac{1}{\gamma_-}$. We have
\begin{lemma}\label{lemma-saddle}
Suppose that $\beta_-<0$, $\alpha_+\ne 0$, $\gamma_->0$ and $V(x^*)\ne 0$ (i.e. $x^*\ne x_L,x_R$). Then the following statements are true.
\begin{enumerate}
\item If $x^*= \frac{1}{\gamma_-}$, then $x_C=0$.
    \item If $x_R<x^*< \frac{1}{\gamma_-}$, then $0< x_C<x_R$.
    \item If $x_L<x^*<x_R$, then $x_R<x_C<-x_L$.
    \item If $x^*<x_L$, then $-x_L<x_C<-x^*$.
\end{enumerate}
\end{lemma}
\begin{proof}[Proof of Lemma \ref{lemma-saddle}]
This follows from elementary calculus using the above expressions for $x_{L,R}$ and 
$ x_C=\sqrt{\frac{\gamma_-x^*-1}{\beta_-}}$. 
\end{proof}

The expressions for $hp(0)$ and $hp'(x)$ are given in \eqref{deriv-hyper}. \\
\textit{Proof of statement 1 of Theorem \ref{theorem-saddle}.} Suppose that $\frac{1}{\gamma_-}<x^*$. Then $hp(0)<0$. Since $x_R<\frac{1}{\gamma_-}$, we have $x_R<x^*$ and $V(x^*)<0$. This, together with \eqref{deriv-hyper}, implies that $hp'(x)<0$ for all $x\ne x^*$. The graph of $hp$ is given in Fig. \ref{fig:contactsaddle}(a) (see the red curve). Since $\frac{1}{\gamma_-}<x^*$, the contact points between the orbits of system \eqref{3degreepolysys} and $y=hp(x)$ are $(x,y)=(x_L,x_R)$ and $(x,y)=(x_R,x_L)$. See the paragraph after \eqref{contactpointsCase1}. 

Since $x_R<x^*$, the domain of $I$ is $[0,x_R[$ (see Fig. \ref{fig:sedlo}(a)). We show that $ I<0$ on $]0,x_R[$. Since $I(0)=0$, it suffices to prove that $ I'<0$ (equivalently, $\tilde I'<0$ or $\Delta<0$) on $]0,x_R[$ (see \eqref{equivfunctions}). We prove that the graph of $hp$ is located below the graph of $\Pi$ for $x\in ]0,x_R[$. Then Lemma \ref{lemma-below-above}.1 will imply that $ I'<0$ on $]0,x_R[$. 

Using $hp(0)<0$ and the paragraph before Theorem \ref{theorem-saddle}, it is clear that the graph of $hp$ lies below the graph of $\Pi$ for $x>0$ and $x\sim 0$ and for $x<x_R$ and $x\sim x_R$. If we assume that there exists an intersection point between the graph of $hp$ and the graph of $\Pi$ for $x\in ]0,x_R[$, then there is a contact point between the orbits of system \eqref{3degreepolysys} and $y=hp(x)$ because \eqref{3degreepolysys} has a saddle at $(x,y)=(0,0)$. The $x-$component of the contact point is contained in $]0,x_R[$. This is in direct contradiction with the fact that $(x,y)=(x_L,x_R)$ and $(x,y)=(x_R,x_L)$ are the only possible contact points. Thus, the graph of $hp$ lies below the graph of $\Pi$ for $x\in ]0,x_R[$. From Theorem \ref{theorem-cyclicity-RHKK}.1 it follows that for any small $\theta>0$, $\cycl(\cup_{x\in[\theta,x_R-\theta]}\Gamma_x)=1$  (the limit cycle is attracting because $I$ is negative).  \\
\\
\textit{Statement 2.} Suppose that $x^*=\frac{1}{\gamma_-}$. Then \eqref{deriv-hyper} implies that $hp(0)=0$. Since $x_R<\frac{1}{\gamma_-}=x^*$, we have $hp'(x)<0$ for all $x\ne x^*$ and the domain of $I$ is $[0,x_R[$ (see the proof of Statement $1$). The graph of $hp$ is given in Fig. \ref{fig:contactsaddle}(b) (the red curve). From Lemma \ref{lemma-saddle}.1 it follows that the contact points between the orbits of system \eqref{3degreepolysys} and $y=hp(x)$ are $(x,y)=(x_L,x_R)$, $(x,y)=(x_R,x_L)$ and $(x,y)=(0,0)$. 

We prove that the graph of $hp$ lies below the graph of $\Pi$ for $x\in ]0,x_R[$. This will imply that $ I<0$ on $]0,x_R[$ (see the proof of Statement $1$). Clearly, the graph of $hp$ lies below the graph of $\Pi$  for $x<x_R$ and $x\sim x_R$. If there is an intersection point between the graph of $hp$ and the graph of $\Pi$ for $x\in ]0,x_R[$, then we have an extra contact point between the orbits of system \eqref{3degreepolysys} and $y=hp(x)$, with the $x-$component contained in $]0,x_R[$. This contact point is different from $(x,y)=(x_L,x_R)$, $(x,y)=(x_R,x_L)$ and $(x,y)=(0,0)$. This gives a contradiction and implies that the graph of $hp$ lies below the graph of $\Pi$ for $x\in ]0,x_R[$. The rest of the statement follows directly from Theorem \ref{theorem-cyclicity-RHKK}.1.    \\
\\
\textit{Statement 3.} Assume that $x_R<x^*<\frac{1}{\gamma_-}$. From \eqref{deriv-hyper} it follows that $hp(0)>0$. Since $x_R<x^*$, we have $hp'(x)<0$ for all $x\ne x^*$ and the domain of $I$ is $[0,x_R[$ (see again the proof of Statement $1$). The graph of $hp$ is given in Fig. \ref{fig:contactsaddle}(c). Lemma \ref{lemma-saddle}.2 implies that the contact points between the orbits of system \eqref{3degreepolysys} and $y=hp(x)$ are $(x,y)=(x_L,x_R)$, $(x,y)=(x_R,x_L)$, $(x,y)=(x_C,-x_C)$ and $(x,y)=(-x_C,x_C)$, with $0<x_C<x_R$. 

First we prove that there is precisely $1$ intersection (counting multiplicity) between the graph of $hp$ and the graph of $\Pi$ for $x\in ]0,x_R[$. This will imply that $I'$ has $1$ zero (counting multiplicity) on $]0,x_R[$. Using Rolle's theorem and $I(0)=0$ we find at most $1$ zero (counting multiplicity) of $I$ on $]0,x_R[$. Then, from Theorem \ref{theorem-cyclicity-RHKK}.2 it follows that for any small $\theta>0$ the set $\cup_{x\in[\theta,x_R-\theta]}\Gamma_x$ can produce at most $2$ limit cycles. The graph of $hp$ lies below the graph of $\Pi$  for $x<x_R$ and $x\sim x_R$ (see the proof of Statement $1$) and, since $hp(0)>0$, the graph of $hp$ lies above the graph of $\Pi$  for $x>0$ and $x\sim 0$. Thus, there exists at least $1$ intersection between the graph of $hp$ and the graph of $\Pi$ for $x\in ]0,x_R[$ (The Intermediate-Value Theorem). If we assume that we have at least $2$ intersections (counting multiplicity), then, besides $(x,y)=(x_C,-x_C)$, we find at least $1$ extra contact point with the $x-$component contained in $]0,x_R[$. This gives a contradiction. Thus, there exists precisely $1$ intersection (counting multiplicity).

Let us prove that $I$ has a (simple) zero in $]0,x_R[$ if $x_R<x^*<\frac{1}{\gamma_-}$ and if $x^*$ is close enough to $\frac{1}{\gamma_-}$. Statement $2$ implies the existence of $x_0\in ]0,x_R[$ such that $I(x_0)<0$ for each $x^*<\frac{1}{\gamma_-}$ and $x^*\sim\frac{1}{\gamma_-}$ ($I$ is continuous). On the other hand, we know that the graph of $hp$ lies above the graph of $\Pi$  for $x>0$ and $x\sim 0$. Then Lemma \ref{lemma-below-above}.1 implies that $I'(x)>0$ for all $x>0$ and $x\sim 0$. Since $I(0)=0$, we have $I(x)>0$ for all $x>0$ and $x\sim 0$. From The Intermediate-Value Theorem it follows now that $I$ has a zero in $]0,x_R[$ when $x^*$ is close enough to $\frac{1}{\gamma_-}$. Then Theorem \ref{theorem-cyclicity-RHKK}.3 implies that for any sufficiently small $\theta>0$, $\cycl(\cup_{x\in[\theta,x_R-\theta]}\Gamma_x)=2$.  \\
\\
\textit{Statement 4.} Suppose that $x^*=x_R$. The domain of $I$ is $[0,x_R[$ (see the proof of Statement $1$). Since $\alpha_+\beta_-\ne 0$ and $V(x^*)=0$, points on the lines $x=x_R$ and $y=x_R$ are solutions of $\overline \Delta=0$ (see Fig. \ref{fig:contactsaddle}(d)). Since the line $y=x_R$ lies above the graph of $\Pi$ for $x\in ]0,x_R[$, Lemma \ref{lemma-below-above}.1 implies that $\tilde I'(x)>0$ (i.e., $I'(x)>0$) for all $x\in ]0,x_R[$. Thus, $I>0$ on $]0,x_R[$. The rest of Statement $4$ follows from Theorem \ref{theorem-cyclicity-RHKK}.1.\\
\\
\textit{Statement 5.} Suppose that $0<x^*<x_R$. Then \eqref{deriv-hyper} and $V(x^*)>0$ imply that $hp(0)>0$ and $hp'(x)>0$ for all $x\ne x^*$. The graph of $hp$ is given in Fig. \ref{fig:contactsaddle}(e). ´

Since $0<x^*<x_R$, the domain of $I$ is $[0,x^*[$ (see Fig. \ref{fig:sedlo}(b)). Clearly, the graph of $hp$ lies above the graph of $\Pi$ for $x\in ]0,x^*[$ and Lemma \ref{lemma-below-above}.1 implies that $I'(x)>0$ for all $x\in ]0,x^*[$. Thus, $I>0$ on $]0,x^*[$. The rest of Statement $5$ follows from Theorem \ref{theorem-cyclicity-RHKK}.1.\\
\\
\textit{Statement 6.} Suppose that $x_L<x^*<0$. From \eqref{deriv-hyper} and $V(x^*)>0$ it follows that $hp(0)<0$ and $hp'(x)>0$ for all $x\ne x^*$. The graph of $hp$ is given in Fig. \ref{fig:contactsaddle}(f). 

Since $x_L<x^*<0$, the domain of $I$ is $[0,\Pi^{-1}(x^*)[$ (Fig. \ref{fig:sedlo}(d)). It is clear that the graph of $hp$ lies below the graph of $\Pi$ for $x\in ]0,\Pi^{-1}(x^*)[$ and  $I'(x)<0$ for all $x\in ]0,\Pi^{-1}(x^*)[$ (see Lemma \ref{lemma-below-above}.1). Thus, $I<0$ on $]0,\Pi^{-1}(x^*)[$. The rest of Statement $6$ follows from Theorem \ref{theorem-cyclicity-RHKK}.1.\\
\\
\textit{Statement 7.} The proof of Statement $7$ is similar to the proof of Statement $4$. Since $x^*=x_L$, the domain of $I$ is $[0,x_R[$ (Fig. \ref{fig:sedlo}(e)).\\
\\
\textit{Statement 8.} Suppose that $x^*<x_L$. From \eqref{deriv-hyper} and $V(x^*)<0$ it follows that $hp(0)<0$ and $hp'(x)<0$ for all $x\ne x^*$. The graph of $hp$ is given in Fig. \ref{fig:contactsaddle}(h). We use Lemma \ref{lemma-saddle}.4 and see that the contact points are $(x,y)=(x_L,x_R)$, $(x,y)=(x_R,x_L)$, $(x,y)=(x_C,-x_C)$ and $(x,y)=(-x_C,x_C)$, with $-x_L<x_C<-x^*$.

Since $x^*<x_L$, the domain of $I$ is $[0,x_R[$ (Fig. \ref{fig:sedlo}(e)). We can prove that the graph of $hp$ lies below the graph of $\Pi$ for $x\in ]0,x_R[$ (see the proof of Statement $1$). Notice that the $x-$coordinate of the above contact points is not contained in $]0,x_R[$.\\
\\
\textit{Statement 9.} Suppose that $\alpha_+=0$. Let us recall that $\beta_-<0$ and $\gamma_->0$. The solutions of $\overline \Delta=0$ are given by \eqref{line-equation} (see the red line in Fig. \ref{fig:contactsaddle}(i)). From \eqref{contactpointsCase2} it follows that the contact points between the orbits of system \eqref{3degreepolysys} and the line given in \eqref{line-equation} are $(x,y)=(x_L,x_R)$ and $(x,y)=(x_R,x_L)$.

The domain of $I$ is $[0,x_R[$ (Fig. \ref{fig:sedlo}(c)). We can show that the line \eqref{line-equation} lies below the graph of $\Pi$ for $x\in ]0,x_R[$ (see the proof of Statement $1$).\end{proof}

\subsection{The node case}\label{node-subsection}

\begin{figure}[htb]
	\begin{center}		\includegraphics[width=12.4cm,height=2.8cm]{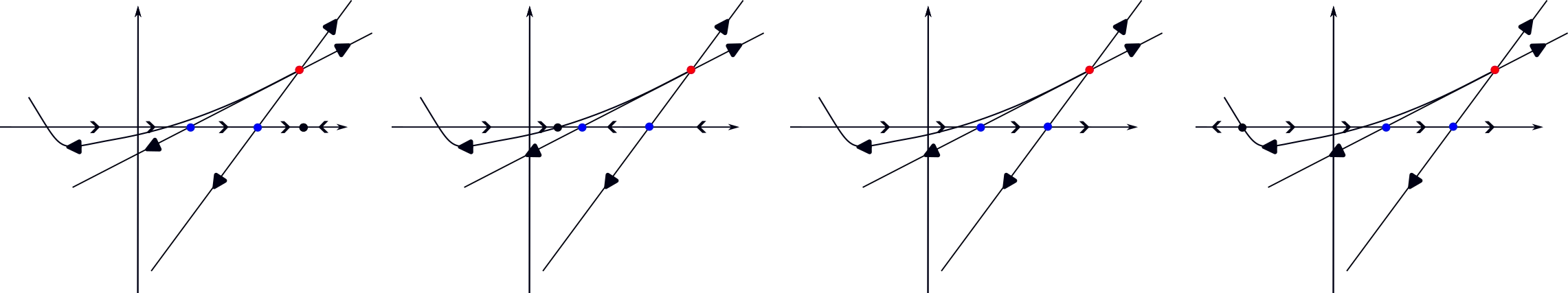}
		{\footnotesize
   \put(-300,38){\footnotesize{$x_R$}}
   \put(-313,38){\footnotesize{$x_L$}}
   \put(-287,38){\footnotesize{$x^*$}}
   \put(-333,-10){\footnotesize{(a) $x_L\le x^*$}}
    \put(-210,38){\footnotesize{$x_R$}}
   \put(-224,38){\footnotesize{$x_L$}}
   \put(-230,48){\footnotesize{$x^*$}}
   \put(-255,-10){\footnotesize{(b) $0<x^*<x_L$}}
    \put(-120,38){\footnotesize{$x_R$}}
   \put(-135,38){\footnotesize{$x_L$}}
   \put(-164,-10){\footnotesize{(c) $\alpha_+=0$}}
    \put(-29,38){\footnotesize{$x_R$}}
   \put(-44,38){\footnotesize{$x_L$}}
   \put(-77,39){\footnotesize{$x^*$}}
   \put(-70,-10){\footnotesize{(d) $x^*<0$}}
   }
         \end{center}
	\caption{Phase portraits of $Z^-$ defined in \eqref{linearPWS} and the direction of the sliding vector field \eqref{sliding-linear} along $y=0$, for $\beta_->0$, $\gamma_->0$ and $\gamma_-^2-4\beta_->0$.  $Z^-$ has a repelling node with distinct eigenvalues. We do not draw the corresponding phase portraits of $Z^+$.}
	\label{fig:cvor}
\end{figure}
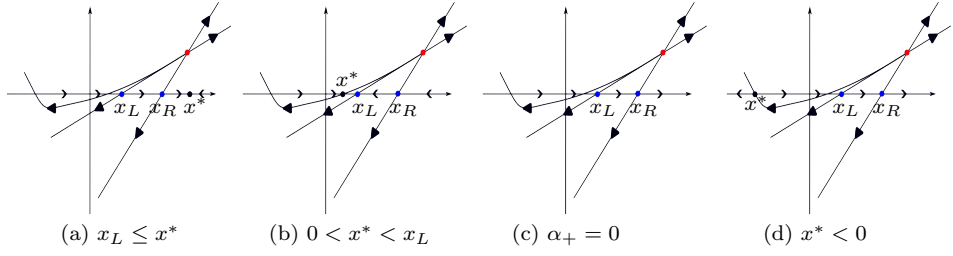

\subsubsection{Distinct eigenvalues} 
In this section we assume that $\beta_->0$, $\gamma_->0$ and $\gamma_-^2-4\beta_->0$. System $Z^-$ has a repelling node at $(x,y)=(\frac{\gamma_-}{\beta_-},\frac{1}{\beta_-})$ with eigenvalues $0<\varkappa_-<\varkappa_+$ where $\varkappa_\pm$ are given in \eqref{eigenvaluesZ-}. The straight-line solution corresponding to the eigenvalue $\varkappa_-$ (resp. $\varkappa_+$) is given by $x=\varkappa_+y+x_L$ (resp. $x=\varkappa_-y+x_R$) where $0<x_L<x_R$ are defined in \eqref{intersect-pointsRL}.
 We refer to Lemma \ref{lemma-LRkappa} and Fig. \ref{fig:cvor}.

Using Fig. \ref{fig:cvor} we see that the domain and image of $\Pi$ are respectively $[0,x_L[$ and $]-\infty,0]$ (see also \cite{Carmona}). The domain of the slow divergence integral $I$ (or $\tilde I$) depends on $x^*$. We distinguish between $4$ cases.
\begin{enumerate}
    \item[(a)] If $\alpha_+<0$ (hence $x^*>0$) and $x_L\le x^*$, then the domain of $I$ is $[0,x_L[$ and we consider the canard cycle $\Gamma_x$ for all $x\in ]0,x_L[$ (see Fig. \ref{fig:cvor}(a)).
    \item[(b)] If $\alpha_+<0$  and $x^*< x_L$, then the domain of $I$ is $[0,x^*[$ and we consider the canard cycle $\Gamma_x$ for all $x\in ]0,x^*[$ (see Fig. \ref{fig:cvor}(b)).
\item[(c)] If $\alpha_+=0$, then we have the same domain of $I$ as in the case (a) (see Fig. \ref{fig:cvor}(c)).
 \item[(d)] If $\alpha_+>0$ (hence $x^*<0$), then the domain of $I$ is $[0,\Pi^{-1}(x^*)[$ and we consider the canard cycle $\Gamma_x$ for all $x\in ]0,\Pi^{-1}(x^*)[$ (see Fig. \ref{fig:cvor}(d)).
\end{enumerate}

Apart from the hyperbolic saddle at the origin, system \eqref{3degreepolysys} has a hyperbolic and attracting node at $(x,y)=(x_L,x_L)$, a hyperbolic and repelling node at $(x,y)=(x_R,x_R)$, and hyperbolic saddles at $(x,y)=(x_R,x_L)$ and $(x,y)=(x_L,x_R)$ (Fig. \ref{fig:contactnode}). Notice that the invariant line $x=x_L$ is the vertical asymptote for the graph of the Poincar\'{e} half-map $\Pi$.

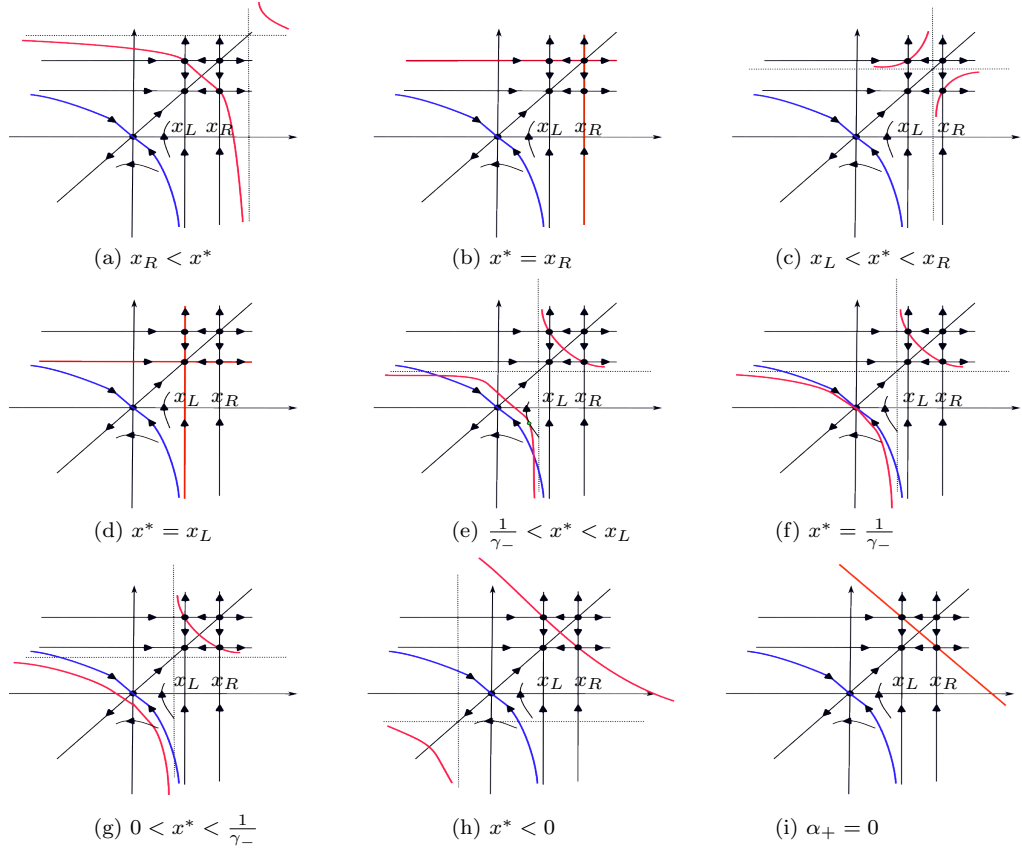
\begin{figure}[htb]
	\begin{center}		\includegraphics[width=13.3cm,height=10.5cm]{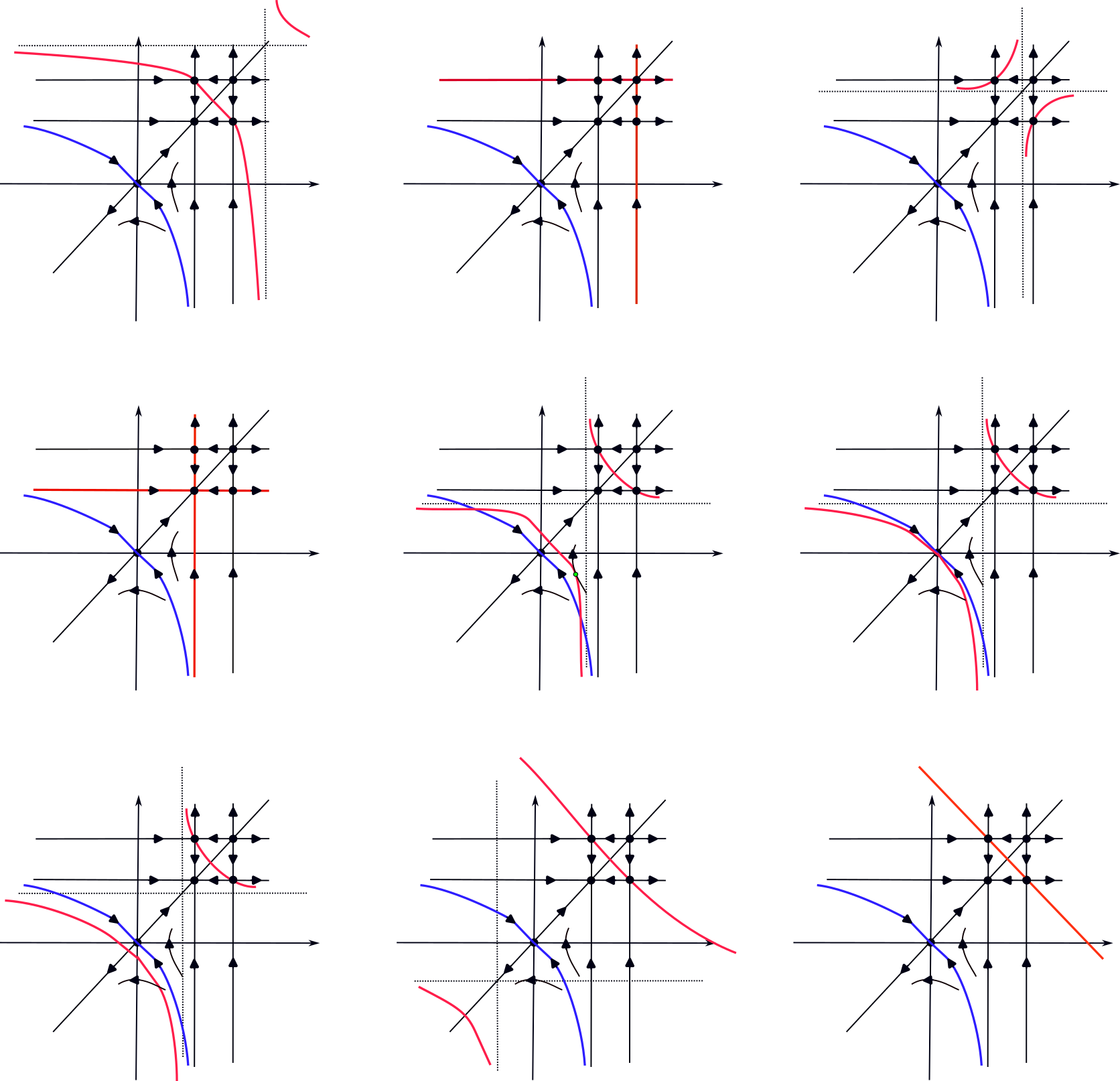}
		{\footnotesize
  \put(-305,250){\footnotesize{$x_R$}}
   \put(-318,250){\footnotesize{$x_L$}}
   \put(-347,200){\footnotesize{(a) $x_R<x^*$}}
   \put(-167,250){\footnotesize{$x_R$}}
   \put(-181,250){\footnotesize{$x_L$}}
   \put(-213,200){\footnotesize{(b) $x^*=x_R$}}
   \put(-31,250){\footnotesize{$x_R$}}
   \put(-46,250){\footnotesize{$x_L$}}
    \put(-92,200){\footnotesize{(c) $x_L<x^*<x_R$}}
    \put(-303,148){\footnotesize{$x_R$}}
   \put(-317,148){\footnotesize{$x_L$}}
    \put(-347,97){\footnotesize{(d) $x^*=x_L$}}
   \put(-167,148){\footnotesize{$x_R$}}
   \put(-179,148){\footnotesize{$x_L$}}
    \put(-213,97){\footnotesize{(e) $\frac{1}{\gamma_-}<x^*<x_L$}}
   \put(-32,148){\footnotesize{$x_R$}}
   \put(-44,148){\footnotesize{$x_L$}}
    \put(-92,97){\footnotesize{(f) $x^*=\frac{1}{\gamma_-}$}}
   \put(-303,41){\footnotesize{$x_R$}}
   \put(-317,41){\footnotesize{$x_L$}}
    \put(-347,-15){\footnotesize{(g) $0<x^*<\frac{1}{\gamma_-}$}}
   \put(-167,41){\footnotesize{$x_R$}}
   \put(-181,41){\footnotesize{$x_L$}}
   \put(-213,-15){\footnotesize{(h) $x^*<0$}}
    \put(-34,41){\footnotesize{$x_R$}}
   \put(-48,41){\footnotesize{$x_L$}}
   \put(-92,-15){\footnotesize{(i) $\alpha_+=0$}}
}
         \end{center}
	\caption{The phase portrait of \eqref{3degreepolysys} for $\beta_->0$, $\gamma_->0$ and $\gamma_-^2-4\beta_->0$, with the curve $\overline \Delta(x,y)=0$ (red). The part of the blue curve located in the fourth quadrant is the graph of $\Pi$. We draw $x=x^*$ and $y=x^*$ using dashed lines. $(x,y)=(x_C,-x_C)$ and $(x,y)=(-x_C,x_C)$ are the intersection points between the red curve and $y=-x$. We draw the contact point $(x,y)=(x_C,-x_C)$ when $x_C$ is positive and contained in the domain of $I$ (Fig. \ref{fig:contactnode}(e)).}
	\label{fig:contactnode}
\end{figure}

\begin{theorem}\label{thm-distinctnode}
Suppose that $\beta_->0$, $\gamma_->0$ and $\gamma_-^2-4\beta_->0$. Then $\frac{1}{\gamma_-}<x_L$ and the following statements are true.
\begin{enumerate}
    \item{($\alpha_+<0$)} If $x_R<x^*$ (Fig. \ref{fig:contactnode}(a)), we have $ I<0$ on $]0,x_L[$ and, for any small $\theta>0$,  $\cycl(\cup_{x\in[\theta,x_L-\theta]}\Gamma_x)=1$ (the limit cycle is attracting).
   \item{($\alpha_+<0$)} If $x^*=x_R$ (Fig. \ref{fig:contactnode}(b)), then we have $I<0$ on $]0,x_L[$ and, for any small $\theta>0$, 
 $\cycl(\cup_{x\in[\theta,x_L-\theta]}\Gamma_x)=1$ (the limit cycle is attracting).
    \item{($\alpha_+<0$)} If $x_L<x^*<x_R$ (Fig. \ref{fig:contactnode}(c)), then we have $ I<0$ on $]0,x_L[$ and, for any small $\theta>0$, $\cycl(\cup_{x\in[\theta,x_L-\theta]}\Gamma_x)=1$. The limit cycle is attracting.
     \item{($\alpha_+<0$)} If $x^*=x_L$ (Fig. \ref{fig:contactnode}(d)), then $I<0$ on $]0,x_L[$ and, for any small $\theta>0$, $\cycl(\cup_{x\in[\theta,x_L-\theta]}\Gamma_x)=1$. The limit cycle is attracting.
     \item{($\alpha_+<0$)} If $\frac{1}{\gamma_-}<x^*<x_L$ (Fig. \ref{fig:contactnode}(e)), then the function $I$ has precisely $1$ zero counting multiplicity on $]0,x^*[$ and, for any sufficiently small $\theta>0$, $\cycl(\cup_{x\in[\theta,x^*-\theta]}\Gamma_x)=2$.
      \item{($\alpha_+<0$)} If $x^*=\frac{1}{\gamma_-}$ (Fig. \ref{fig:contactnode}(f)), then $I>0$ on $]0,x^*[$ and, for any small $\theta>0$, $\cycl(\cup_{x\in[\theta,x^*-\theta]}\Gamma_x)=1$ (the limit cycle is repelling).
      \item{($\alpha_+<0$)} If $0<x^*<\frac{1}{\gamma_-}$ (Fig. \ref{fig:contactnode}(g)), then $I>0$ on $]0,x^*[$ and, for any small $\theta>0$, $\cycl(\cup_{x\in[\theta,x^*-\theta]}\Gamma_x)=1$ (the limit cycle is repelling).
        \item{($\alpha_+>0$)} If $x^*<0$ (Fig. \ref{fig:contactnode}(h)), then $ I<0$ on $]0,\Pi^{-1}(x^*)[$ and, for any small $\theta>0$, $\cycl(\cup_{x\in[\theta,\Pi^{-1}(x^*)-\theta]}\Gamma_x)=1$ (the limit cycle is attracting).
    \item If $\alpha_+=0$ (Fig. \ref{fig:contactnode}(i)), then we have $I<0$ on $]0,x_L[$ and, for any small $\theta>0$,  $\cycl(\cup_{x\in[\theta,x_L-\theta]}\Gamma_x)=1$ (the limit cycle is attracting).
\end{enumerate}
\end{theorem}

\begin{proof}
Suppose that $\beta_->0$, $\gamma_->0$ and $\gamma_-^2-4\beta_->0$. From \eqref{intersect-pointsRL} it follows that
$$x_L=\frac{2}{\gamma_-+\sqrt{\gamma_-^2-4 \beta_-}}>0, \ \ x_R=\frac{2}{\gamma_--\sqrt{\gamma_-^2-4 \beta_-}}>0,$$
with $x_L<x_R$, $V(x_L)=V(x_R)=0$ where $V(x)=\beta_-x^2-\gamma_-x+1$. The graph of $V$ is concave up ($\beta_->0$).
It is not difficult to see that $\frac{1}{\gamma_-}<x_L$. Using \eqref{contact-important} $x_C\ge 0$ is well-defined if $x^*\ge\frac{1}{\gamma_-}$.

    \begin{lemma}\label{lemma-nodedistinct}
Suppose that $\beta_->0$, $\gamma_->0$, $\gamma_-^2-4\beta_->0$, $\alpha_+\ne 0$ and $V(x^*)\ne 0$ (i.e. $x^*\ne x_L,x_R$). Then the following statements are true.
\begin{enumerate}
\item If $x_R<x^*$, then $x_R<x_C<x^*$.
\item If $x_L<x^*<x_R$, then $x^*<x_C<x_R$.
    \item If $\frac{1}{\gamma_-}<x^*<x_L $, then $0< x_C<x^*$.
    \item If $x^*= \frac{1}{\gamma_-}$, then $x_C=0$.
\end{enumerate}
\end{lemma}
Now, we prove the statements of Theorem \ref{thm-distinctnode}.\\
\\
\textit{Statement 1.} Suppose that $x_R<x^*$. Then \eqref{deriv-hyper} and $V(x^*)>0$ imply that $hp(0)>0$ and $hp'(x)<0$ for all $x\ne x^*$. The graph of $hp$ is given in Fig. \ref{fig:contactnode}(a). 

Since $x_L<x_R<x^*$, the domain of $I$ is $[0,x_L[$ (see Fig. \ref{fig:cvor}(a)). It is clear (Fig. \ref{fig:contactnode}(a)) that the graph of $hp$ lies above the graph of $\Pi$ for $x\in ]0,x_L[$ and Lemma \ref{lemma-below-above}.2 implies that $I'(x)<0$ for all $x\in ]0,x_L[$. Hence, $I<0$ on $]0,x_L[$. Following Theorem \ref{theorem-cyclicity-RHKK}.1, for any small $\theta>0$, $\cycl(\cup_{x\in[\theta,x_L-\theta]}\Gamma_x)=1$ (the limit cycle is attracting because $I$ is negative).\\
\\
\textit{Statement 2.} Assume that $x^*=x_R$. The domain of $I$ is $[0,x_L[$ (see the proof of Statement $1$). Since $\alpha_+\beta_-\ne 0$ and $V(x^*)=0$, $\overline \Delta=0$ is the union of $x=x_R$ and $y=x_R$ (see Fig. \ref{fig:contactnode}(b)). The horizontal line $y=x_R$ lies above the graph of $\Pi$ for $x\in ]0,x_L[$, and Lemma \ref{lemma-below-above}.2 implies that $I'(x)<0$ for all $x\in ]0,x_L[$. Thus, $I<0$ on $]0,x_L[$. For any small $\theta>0$, $\cycl(\cup_{x\in[\theta,x_L-\theta]}\Gamma_x)=1$ (the limit cycle is attracting). See Theorem \ref{theorem-cyclicity-RHKK}.1.\\
\\
\textit{Statement 3.} Suppose that $x_L<x^*<x_R$. Then \eqref{deriv-hyper} and $V(x^*)<0$ imply that $hp(0)>0$ and $hp'(x)>0$ for all $x\ne x^*$. The graph of $hp$ is given in Fig. \ref{fig:contactnode}(c). 
The proof of Statement $3$ is similar to the proof of Statement $1$.\\
\\
\textit{Statement 4.} Statement $4$ can be proved in the same fashion as Statement $2$ (see Fig. \ref{fig:contactnode}(d)).\\
\\
\textit{Statement 5.} Suppose that $\frac{1}{\gamma_-}<x^*<x_L$. From \eqref{deriv-hyper} and $V(x^*)>0$ it follows that $hp(0)>0$ and $hp'(x)<0$ for all $x\ne x^*$.  The graph of $hp$ is given in Fig. \ref{fig:contactnode}(e). Lemma \ref{lemma-nodedistinct}.3 implies that the contact points between the orbits of system \eqref{3degreepolysys} and $y=hp(x)$ are $(x,y)=(x_L,x_R)$, $(x,y)=(x_R,x_L)$, $(x,y)=(x_C,-x_C)$ and $(x,y)=(-x_C,x_C)$, with $0<x_C<x^*$.

The domain of $I$ is $[0,x^*[$ (Fig. \ref{fig:cvor}(b)). First we show that there is precisely $1$ intersection (counting multiplicity) between the graph of $hp$ and the graph of $\Pi$ for $x\in ]0,x^*[$. This will imply that $I$ has at most $1$ zero (counting multiplicity)  on $]0,x^*[$ (see the proof of Theorem \ref{theorem-saddle}.3).   
Since $hp(0)>0$, the graph of $hp$ lies above the graph of $\Pi$  for $x>0$ and $x\sim 0$. Notice that $\Pi(x^*)$ is finite and that $hp(x)\to -\infty$ as $x\to x^*-$. Thus, the graph of $hp$ lies below the graph of $\Pi$  for $x<x^*$ and $x\sim x^*$. We conclude that there exists at least $1$ intersection between the graph of $hp$ and the graph of $\Pi$ for $x\in ]0,x^*[$ (The Intermediate-Value Theorem).
If we suppose that there exist at least $2$ intersections (counting multiplicity), then, apart from $(x,y)=(x_C,-x_C)$, we have at least $1$ extra contact point with the $x-$coordinate contained in $]0,x^*[$. This gives a contradiction. Thus, we have found precisely $1$ intersection (counting multiplicity).

Now, we prove that $I$ has a zero in $]0,x^*[$. Then  the above discussion implies that $I$ has precisely $1$ zero counting multiplicity on $]0,x^*[$. Since the graph of $hp$ lies above the graph of $\Pi$  for $x>0$ and $x\sim 0$, Lemma \ref{lemma-below-above}.2 implies that $I'(x)<0$ for $x>0$ and $x\sim 0$. Hence, $I(x)<0$ for $x>0$ and $x\sim 0$. The integral in \eqref{SDI-linearPWS} can be written as
$$\int_{\Pi(x)}^{x}\frac{udu}{X^{sl}(u)}=\int_{\Pi(x)}^{0}\frac{udu}{X^{sl}(u)}+\int_{0}^{x}\frac{udu}{X^{sl}(u)}.$$
Since $\Pi(x^*)<0$ (finite) and $x^*=-\frac{B-\delta_+}{\alpha_+}>0$, the first component $\int_{\Pi(x)}^{0}$ tends to a negative number as $x\to x^*-$ (thus, it is bounded). It is clear that the second integral is divergent: $\int_{0}^{x}\to +\infty$ as $x\to x^*-$. This implies that $I$ is positive for $x<x^*$ and $x\sim x^*$. From The Intermediate-Value Theorem it follows that $I$ has a zero in $]0,x^*[$.

Now, Theorem \ref{theorem-cyclicity-RHKK}.3 implies that for any sufficiently small $\theta>0$ we get $\cycl(\cup_{x\in[\theta,x^*-\theta]}\Gamma_x)=2$.\\
\\
\textit{Statement 6.} Suppose that $x^*=\frac{1}{\gamma_-}$. Then \eqref{deriv-hyper} and $V(x^*)>0$ imply that $hp(0)=0$ and $hp'(x)<0$ for all $x\ne x^*$. The graph of $hp$ is given in Fig. \ref{fig:contactnode}(f). Using Lemma \ref{lemma-nodedistinct}.4, the contact points between the orbits of system \eqref{3degreepolysys} and $y=hp(x)$ are $(x,y)=(x_L,x_R)$, $(x,y)=(x_R,x_L)$ and $(x,y)=(0,0)$.

The domain of $I$ is $[0,x^*[$ (see the proof of Statement $5$). 
We can prove that the graph of $hp$ lies below the graph of $\Pi$ for $x\in ]0,x^*[$ using the same idea as in the proof of Theorem  \ref{theorem-saddle}.2. Then Lemma \ref{lemma-below-above}.2 implies that $ I>0$ on $]0,x^*[$. Following Theorem \ref{theorem-cyclicity-RHKK}.1, for any small $\theta>0$, $\cycl(\cup_{x\in[\theta,x^*-\theta]}\Gamma_x)=1$ (the limit cycle is repelling).\\
\\
\textit{Statement 7.} Suppose that $0<x^*<\frac{1}{\gamma_-}$. Then we have $hp(0)<0$ and $hp'(x)<0$ for all $x\ne x^*$. The graph of $hp$ is given in Fig. \ref{fig:contactnode}(g). The contact points between the orbits of system \eqref{3degreepolysys} and $y=hp(x)$ are $(x,y)=(x_L,x_R)$ and $(x,y)=(x_R,x_L)$.

The domain of $I$ is $[0,x^*[$ (see the proof of Statement $5$). 
Again, we can show that the graph of $hp$ lies below the graph of $\Pi$ for $x\in ]0,x^*[$ using the same technique as in the proof of Theorem  \ref{theorem-saddle}.1. Then Lemma \ref{lemma-below-above}.2 implies that $ I>0$ on $]0,x^*[$. Using Theorem \ref{theorem-cyclicity-RHKK}.1, for any small $\theta>0$, we get $\cycl(\cup_{x\in[\theta,x^*-\theta]}\Gamma_x)=1$ (the limit cycle is repelling).\\
\\
\textit{Statement 8.} Suppose that $x^*<0$. Then $hp(0)>0$ and $hp'(x)<0$ for all $x\ne x^*$. The graph of $hp$ is given in Fig. \ref{fig:contactnode}(h).

The domain of $I$ is $[0,\Pi^{-1}(x^*)[$ (Fig. \ref{fig:cvor}(d)). Clearly, 
the graph of $hp$ lies above the graph of $\Pi$ for $x\in ]0,\Pi^{-1}(x^*)[$.  Then Lemma \ref{lemma-below-above}.2 implies that $ I<0$ on $]0,\Pi^{-1}(x^*)[$. Again, Theorem \ref{theorem-cyclicity-RHKK}.1 implies that for any small $\theta>0$ $\cycl(\cup_{x\in[\theta,\Pi^{-1}(x^*)-\theta]}\Gamma_x)=1$ (the limit cycle is attracting).\\
\\
\textit{Statement 9.} Assume that $\alpha_+=0$. Recall that $\beta_->0$ and $\gamma_->0$. The curve $\overline \Delta=0$ is given by \eqref{line-equation}. We refer to Fig. \ref{fig:contactnode}(i).
The domain of $I$ is $[0,x_L[$ (Fig. \ref{fig:cvor}(c)). The graph of $hp$ lies above the graph of $\Pi$ for $x\in ]0,x_L[$. This, together with Lemma \ref{lemma-below-above}.2 and Theorem \ref{theorem-cyclicity-RHKK}.1, implies Statement $9$.\end{proof}

\subsubsection{Repeated eigenvalues} 
In this section we assume that $\beta_->0$, $\gamma_->0$ and $\gamma_-^2-4\beta_-=0$. System $Z^-$ has a repelling node at $(x,y)=(\frac{4}{\gamma_-},\frac{4}{\gamma_-^2})$ with repeated eigenvalues $\varkappa_\pm=\frac{\gamma_-}{2}$. The straight-line solution corresponding to the eigenvalue is given by $x=\frac{\gamma_-}{2}y+\frac{2}{\gamma_-}$.
 We refer to Lemma \ref{lemma-LRkappa} and Fig. \ref{fig:cvoradded}.

\begin{figure}[htb]
	\begin{center}		\includegraphics[width=12.4cm,height=3.05cm]{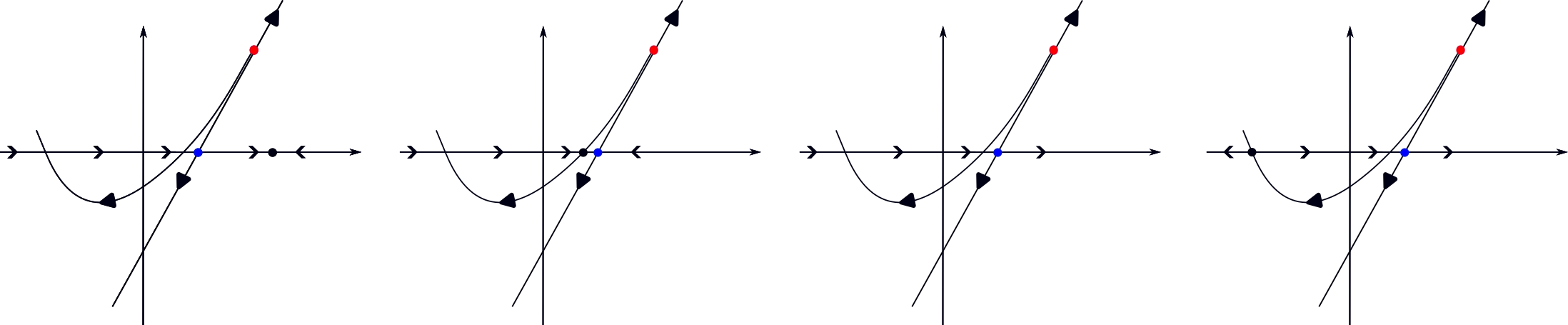}
		{\footnotesize
   \put(-315,36){\footnotesize{ $\frac{2}{\gamma_-}$}}
   \put(-295,39){\footnotesize{$x^*$}}
   \put(-335,-12){\footnotesize{(a) $\frac{2}{\gamma_-}\le x^*$}}
   \put(-224,37){\footnotesize{ $\frac{2}{\gamma_-}$}}
   \put(-225,49){\footnotesize{$x^*$}}
   \put(-134,37){\footnotesize{ $\frac{2}{\gamma_-}$}}
   \put(-42,37){\footnotesize{ $\frac{2}{\gamma_-}$}}
    \put(-75,39){\footnotesize{$x^*$}}
    \put(-257,-12){\footnotesize{(b) $0<x^*<\frac{2}{\gamma_-}$}}
    \put(-164,-12){\footnotesize{(c) $\alpha_+=0$}}
     \put(-70,-12){\footnotesize{(d) $x^*<0$}}
}
         \end{center}
	\caption{Phase portraits of $Z^-$ defined in \eqref{linearPWS} and the direction of the sliding vector field \eqref{sliding-linear} along $y=0$, with $\gamma_->0$ and $\gamma_-^2-4\beta_-=0$. $Z^-$ has a repelling node with repeated eigenvalues.  We do not draw the corresponding phase portraits of $Z^+$.}
	\label{fig:cvoradded}
\end{figure}

In this case the domain and image of $\Pi$ are respectively $[0,\frac{2}{\gamma_-}[$ and $]-\infty,0]$ (see also \cite{Carmona}). We distinguish between $4$ cases.
\begin{enumerate}
    \item[(a)] If $\alpha_+<0$ (hence $x^*>0$) and $\frac{2}{\gamma_-}\le x^*$, then the domain of $I$ is $[0,\frac{2}{\gamma_-}[$ and we consider the canard cycle $\Gamma_x$ for all $x\in ]0,\frac{2}{\gamma_-}[$ (see Fig. \ref{fig:cvoradded}(a)).
    \item[(b)] If $\alpha_+<0$  and $x^*< \frac{2}{\gamma_-}$, then the domain of $I$ is $[0,x^*[$ and we consider the canard cycle $\Gamma_x$ for all $x\in ]0,x^*[$ (see Fig. \ref{fig:cvoradded}(b)).
\item[(c)] If $\alpha_+=0$, then we have the same domain of $I$ as in the case (a) (see Fig. \ref{fig:cvoradded}(c)).
 \item[(d)] If $\alpha_+>0$ (hence $x^*<0$), then the domain of $I$ is $[0,\Pi^{-1}(x^*)[$ and we consider the canard cycle $\Gamma_x$ for all $x\in ]0,\Pi^{-1}(x^*)[$ (see Fig. \ref{fig:cvoradded}(d)).
\end{enumerate}

System \eqref{3degreepolysys} has a hyperbolic saddle at the origin and a singularity at $(x,y)=(\frac{2}{\gamma_-},\frac{2}{\gamma_-})$ which is linearly zero (for more details see Fig. \ref{fig:contactrepnode}). The graph of the Poincar\'{e} half-map $\Pi$ approaches the invariant line $x=\frac{2}{\gamma_-}$.

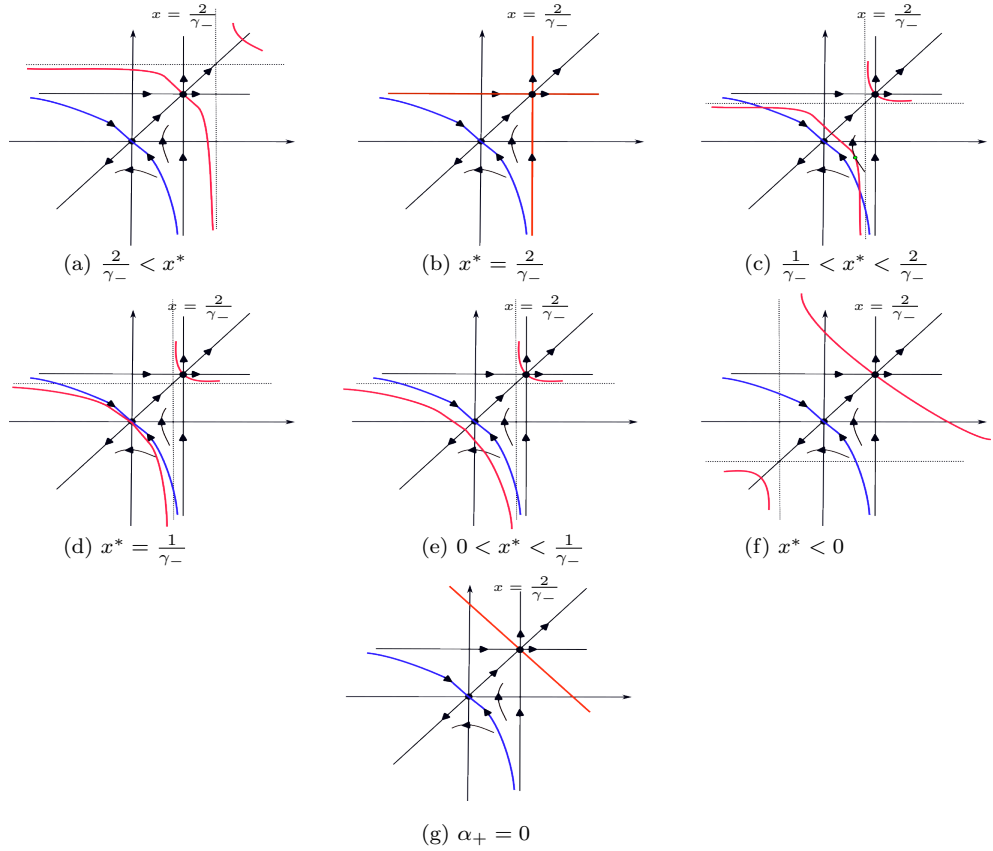
\begin{figure}[htb]
	\begin{center}		\includegraphics[width=12.9cm,height=10.5cm]{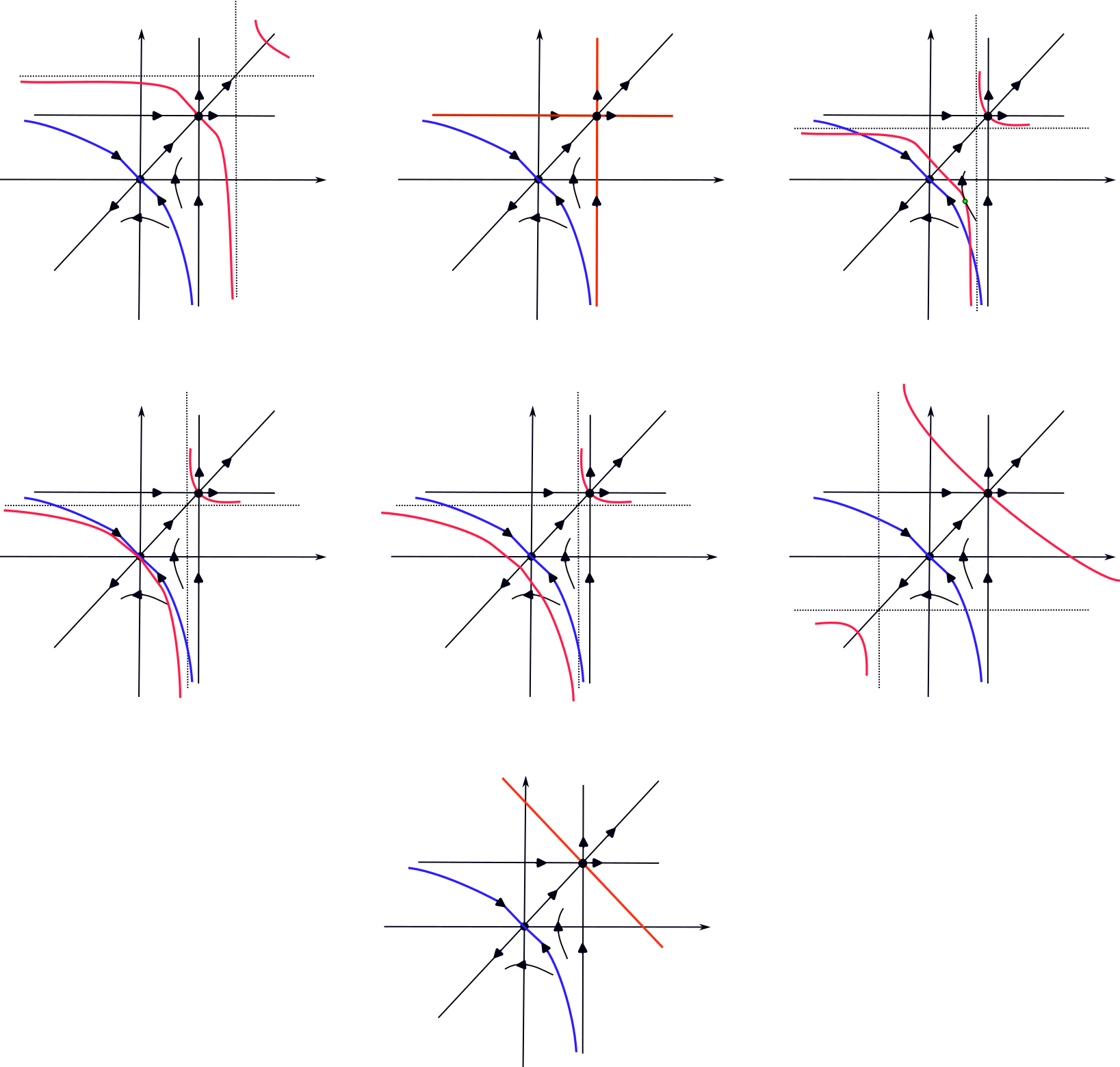}
		{\footnotesize
  \put(-314,295){\tiny{$x=\frac{2}{\gamma_-}$}}
   \put(-347,200){\footnotesize{(a) $\frac{2}{\gamma_-}<x^*$}}
   \put(-181,295){\tiny{$x=\frac{2}{\gamma_-}$}}
   \put(-213,200){\footnotesize{(b) $x^*=\frac{2}{\gamma_-}$}}
   \put(-50,295){\tiny{$x=\frac{2}{\gamma_-}$}}
    \put(-92,200){\footnotesize{(c) $\frac{1}{\gamma_-}<x^*<\frac{2}{\gamma_-}$}}
    \put(-308,185){\tiny{$x=\frac{2}{\gamma_-}$}}
    \put(-347,93){\footnotesize{(d) $x^*=\frac{1}{\gamma_-}$}}
   \put(-182,185){\tiny{$x=\frac{2}{\gamma_-}$}}
    \put(-213,93){\footnotesize{(e) $0<x^*<\frac{1}{\gamma_-}$}}
   \put(-51,185){\tiny{$x=\frac{2}{\gamma_-}$}}
    \put(-92,93){\footnotesize{(f) $x^*<0$}}
   \put(-186,79){\tiny{$x=\frac{2}{\gamma_-}$}}
   \put(-213,-15){\footnotesize{(g) $\alpha_+=0$}}
}
         \end{center}
	\caption{The phase portrait of \eqref{3degreepolysys} for $\beta_->0$, $\gamma_->0$ and $\gamma_-^2-4\beta_-=0$, with the curve $\overline \Delta(x,y)=0$ (red). The part of the blue curve located in the fourth quadrant is the graph of $\Pi$. We draw $x=x^*$ and $y=x^*$ using dashed lines. We draw the contact point $(x,y)=(x_C,-x_C)$ when $x_C$ is positive and contained in the domain of $I$ (Fig. \ref{fig:contactrepnode}(c)).}
	\label{fig:contactrepnode}
\end{figure}

\begin{theorem}\label{thm-repeatedeigen}
Suppose that $\beta_->0$, $\gamma_->0$ and $\gamma_-^2-4\beta_-=0$. Then the following statements are true.
\begin{enumerate}
    \item{($\alpha_+<0$)} If $\frac{2}{\gamma_-}<x^*$ (Fig. \ref{fig:contactrepnode}(a)), then $ I<0$ on $]0,\frac{2}{\gamma_-}[$ and, for any small $\theta>0$, $\cycl(\cup_{x\in[\theta,\frac{2}{\gamma_-}-\theta]}\Gamma_x)=1$ (the limit cycle is attracting).
   \item{($\alpha_+<0$)} If $x^*=\frac{2}{\gamma_-}$ (Fig. \ref{fig:contactrepnode}(b)), then $I<0$ on $]0,\frac{2}{\gamma_-}[$ and, for any small $\theta>0$, $\cycl(\cup_{x\in[\theta,\frac{2}{\gamma_-}-\theta]}\Gamma_x)=1$ (the limit cycle is attracting).
     \item{($\alpha_+<0$)} If $\frac{1}{\gamma_-}<x^*<\frac{2}{\gamma_-}$ (Fig. \ref{fig:contactrepnode}(c)), then the function $I$ has precisely $1$ zero counting multiplicity on $]0,x^*[$ and, for any sufficiently small $\theta>0$, $\cycl(\cup_{x\in[\theta,x^*-\theta]}\Gamma_x)=2$.
      \item{($\alpha_+<0$)} If $x^*=\frac{1}{\gamma_-}$ (Fig. \ref{fig:contactrepnode}(d)), then $I>0$ on $]0,x^*[$ and, for any small $\theta>0$, $\cycl(\cup_{x\in[\theta,x^*-\theta]}\Gamma_x)=1$ (the limit cycle is repelling).
      \item{($\alpha_+<0$)} If $0<x^*<\frac{1}{\gamma_-}$ (Fig. \ref{fig:contactrepnode}(e)), then $I>0$ on $]0,x^*[$ and, for any small $\theta>0$, $\cycl(\cup_{x\in[\theta,x^*-\theta]}\Gamma_x)=1$ (the limit cycle is repelling).
        \item{($\alpha_+>0$)} If $x^*<0$ (Fig. \ref{fig:contactrepnode}(f)), then $ I<0$ on $]0,\Pi^{-1}(x^*)[$ and, for any small $\theta>0$, $\cycl(\cup_{x\in[\theta,\Pi^{-1}(x^*)-\theta]}\Gamma_x)=1$ (the limit cycle is attracting).
         \item If $\alpha_+=0$ (Fig. \ref{fig:contactrepnode}(g)), then $I<0$ on $]0,\frac{2}{\gamma_-}[$ and, for any small $\theta>0$, $\cycl(\cup_{x\in[\theta,\frac{2}{\gamma_-}-\theta]}\Gamma_x)=1$ (the limit cycle is attracting).
    
\end{enumerate}
\end{theorem}

\begin{proof}
Let $\beta_->0$, $\gamma_->0$ and $\gamma_-^2-4\beta_-=0$. From \eqref{intersect-pointsRL} it follows that
$x_L=x_R=\frac{2}{\gamma_-}$ and $V(\frac{2}{\gamma_-})=0$. The graph of $V$ is concave up and $V(x)>0$ for all $x\ne \frac{2}{\gamma_-}$.
 Using \eqref{contact-important} and $\gamma_-^2-4\beta_-=0$ we have $x_C=\frac{2}{\gamma_-}\sqrt{\gamma_-x^*-1}$.

  \begin{lemma}\label{lemma-noderepeated}
Suppose that $\beta_->0$, $\gamma_->0$, $\gamma_-^2-4\beta_-=0$, $\alpha_+\ne 0$ and $V(x^*)\ne 0$ (i.e. $x^*\ne \frac{2}{\gamma_-}$). Then the following statements are true.
\begin{enumerate}
\item If $\frac{2}{\gamma_-}<x^*$, then $\frac{2}{\gamma_-}<x_C<x^*$.
    \item If $\frac{1}{\gamma_-}<x^*<\frac{2}{\gamma_-} $, then $0< x_C<x^*$.
    \item If $x^*= \frac{1}{\gamma_-}$, then $x_C=0$.
\end{enumerate}
\end{lemma}
Now, we prove the statements of Theorem \ref{thm-repeatedeigen}.\\
\\
\textit{Statement 1.} Suppose that $\frac{2}{\gamma_-}<x^*$. From \eqref{deriv-hyper} it follows that $hp(0)>0$ and $hp'(x)<0$ for all $x\ne x^*$. The graph of $hp$ is given in Fig. \ref{fig:contactrepnode}(a). 
Since $\frac{2}{\gamma_-}<x^*$, the domain of $I$ is $[0,\frac{2}{\gamma_-}[$ (see Fig. \ref{fig:cvoradded}(a)). The proof now proceeds in a similar fashion to the proof of Theorem \ref{thm-distinctnode}.1.  \\
\\
\textit{Statement 2.}  Suppose that $x^*=\frac{2}{\gamma_-}$. The domain of $I$ is $[0,\frac{2}{\gamma_-}[$ (see Fig. \ref{fig:cvoradded}(a)). The proof of Statement 2 is similar to the proof of Theorem \ref{thm-distinctnode}.2 or Theorem \ref{thm-distinctnode}.4.\\
\\
\textit{Statement 3.} Suppose that $\frac{1}{\gamma_-}<x^*<\frac{2}{\gamma_-}$. From \eqref{deriv-hyper} it follows that $hp(0)>0$ and $hp'(x)<0$ for all $x\ne x^*$. The graph of $hp$ is given in Fig. \ref{fig:contactrepnode}(c). 
The domain of $I$ is $[0,x^*[$ (see Fig. \ref{fig:cvoradded}(b)). The proof is now analogous to the proof of Theorem \ref{thm-distinctnode}.5. Instead of Lemma \ref{lemma-nodedistinct}.3 we use Lemma \ref{lemma-noderepeated}.2 and find the following contact points: $(x,y)=(\frac{2}{\gamma_-},\frac{2}{\gamma_-})$, $(x,y)=(x_C,-x_C)$ and $(x,y)=(-x_C,x_C)$, with $0<x_C<x^*$.  \\
\\
\textit{Statement 4.} Suppose that $x^*=\frac{1}{\gamma_-}$. We have $hp(0)=0$ and $hp'(x)<0$ for all $x\ne x^*$ (see Fig. \ref{fig:contactrepnode}(d)). The domain of $I$ is $[0,x^*[$ (see Fig. \ref{fig:cvoradded}(b)). The proof is similar to the proof of Theorem \ref{thm-distinctnode}.6. Using Lemma \ref{lemma-noderepeated}.3 the contact points are $(x,y)=(\frac{2}{\gamma_-},\frac{2}{\gamma_-})$ and $(x,y)=(0,0)$.  \\
\\
\textit{Statement 5.} Suppose that $0<x^*<\frac{1}{\gamma_-}$. We have $hp(0)<0$ and $hp'(x)<0$ for all $x\ne x^*$ (see Fig. \ref{fig:contactrepnode}(e)). The domain of $I$ is $[0,x^*[$ (see Fig. \ref{fig:cvoradded}(b)). The proof is similar to the proof of Theorem \ref{thm-distinctnode}.7. We have $1$ contact point: $(x,y)=(\frac{2}{\gamma_-},\frac{2}{\gamma_-})$.  \\
\\
\textit{Statement 6.} Suppose that $x^*<0$. We have $hp(0)>0$ and $hp'(x)<0$ for all $x\ne x^*$ (see Fig. \ref{fig:contactrepnode}(f)). The domain of $I$ is $[0,\Pi^{-1}(x^*)[$ (see Fig. \ref{fig:cvoradded}(d)). The proof is analogous to the proof of Theorem \ref{thm-distinctnode}.8.  \\
\\
\textit{Statement 7.} Suppose that $\alpha_+=0$. The curve $\overline \Delta=0$ is given by \eqref{line-equation}: $y=\frac{4}{\gamma_-}-x$. We refer to Fig. \ref{fig:contactrepnode}(g).
The domain of $I$ is $[0,\frac{2}{\gamma_-}[$ (Fig. \ref{fig:cvoradded}(c)). The proof is analogous to the proof of Theorem \ref{thm-distinctnode}.9. 
\end{proof}

\subsection{The focus case}\label{focus-subsection}

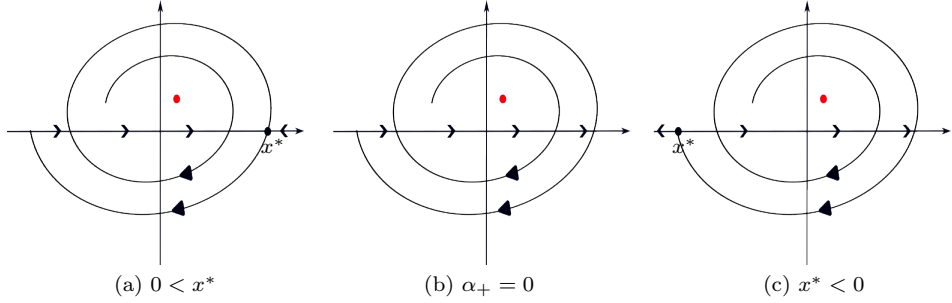
\begin{figure}[htb]
	\begin{center}		\includegraphics[width=12.4cm,height=3.5cm]{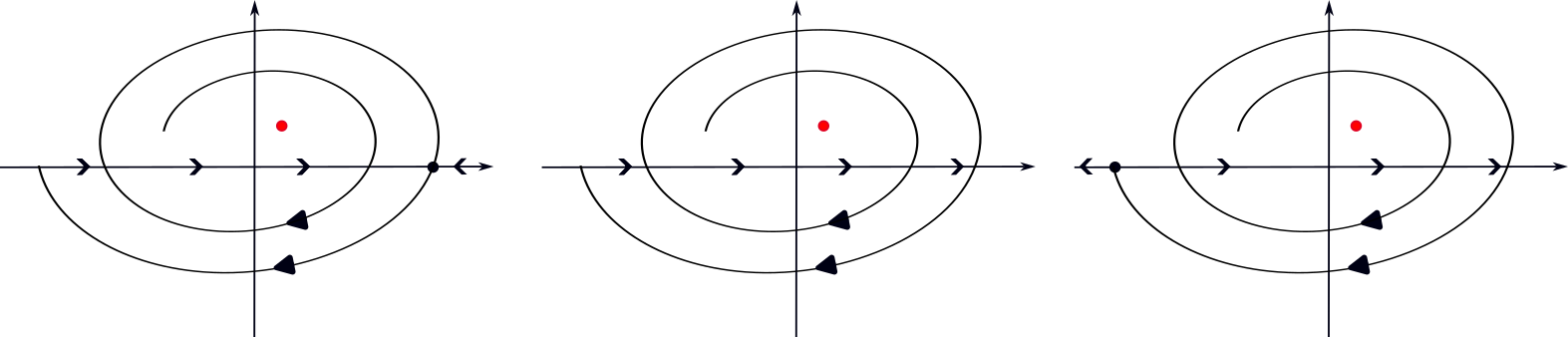}
		{\footnotesize
   \put(-258,42){\footnotesize{$x^*$}}
   \put(-313,-10){\footnotesize{(a) $0<x^*$}}
   \put(-197,-10){\footnotesize{(b) $\alpha_+=0$}}
   \put(-104,42){\footnotesize{$x^*$}}
   \put(-70,-10){\footnotesize{(c) $x^*<0$}}
}
         \end{center}
	\caption{Phase portraits of $Z^-$ defined in \eqref{linearPWS} and the direction of the sliding vector field \eqref{sliding-linear} along $y=0$, for $\beta_->0$, $\gamma_->0$ and $\gamma_-^2-4\beta_-<0$. $Z^-$ has a repelling focus. We do not draw the corresponding phase portraits of $Z^+$.}
	\label{fig:foc-cent}
\end{figure}

Here we suppose that $\beta_->0$, $\gamma_->0$ and $\gamma_-^2-4\beta_-<0$. System $Z^-$ has a repelling focus at $(x,y)=(\frac{\gamma_-}{\beta_-},\frac{1}{\beta_-})$.
 We refer to Lemma \ref{lemma-LRkappa} and Fig. \ref{fig:foc-cent}.
The domain and image of $\Pi$ are respectively $[0,+\infty[$ and $]-\infty,0]$ (see also \cite{Carmona}). The domain of the slow divergence integral $I$ (or $\tilde I$) depends on $x^*$ and we have $3$ cases.
\begin{enumerate}
    \item[(a)] If $\alpha_+<0$ (hence $x^*>0$), then the domain of $I$ is $[0,x^*[$ and we consider the canard cycle $\Gamma_x$ for all $x\in ]0,x^*[$ (see Fig. \ref{fig:foc-cent}(a)).
    \item[(b)] If $\alpha_+=0$, then the domain of $I$ is $[0,+\infty[$ and we consider the canard cycle $\Gamma_x$ for all $x\in ]0,+\infty[$ (see Fig. \ref{fig:foc-cent}(b)).
\item[(c)]  If $\alpha_+>0$ (hence $x^*<0$), then the domain of $I$ is $[0,\Pi^{-1}(x^*)[$ and we consider the canard cycle $\Gamma_x$ for all $x\in ]0,\Pi^{-1}(x^*)[$ (see Fig. \ref{fig:foc-cent}(c)).
\end{enumerate}

System \eqref{3degreepolysys} has a hyperbolic saddle at the origin (Fig. \ref{fig:contactfocus}). 

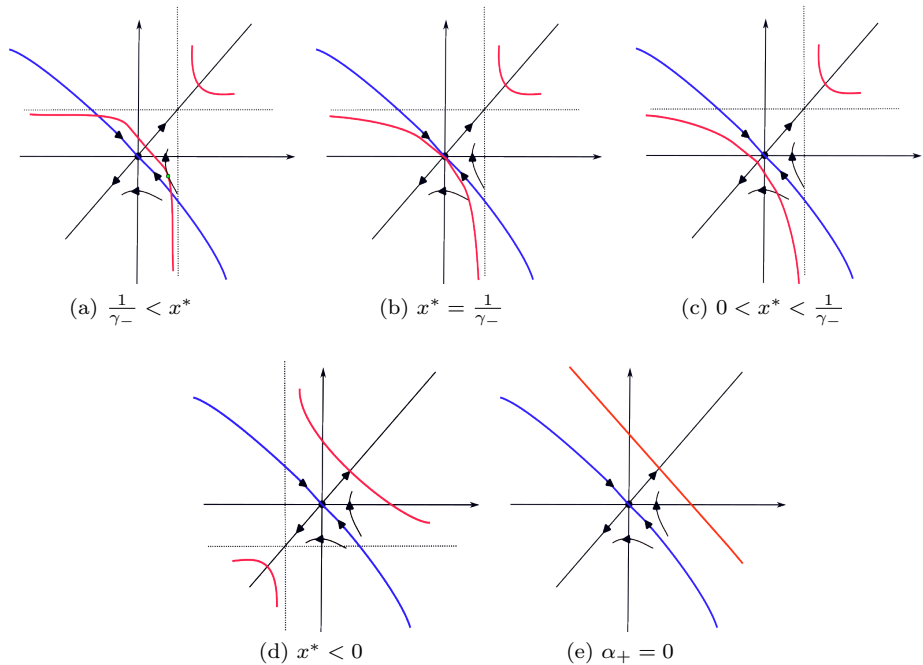
\begin{figure}[htb]
	\begin{center}		\includegraphics[width=12cm,height=8.3cm]{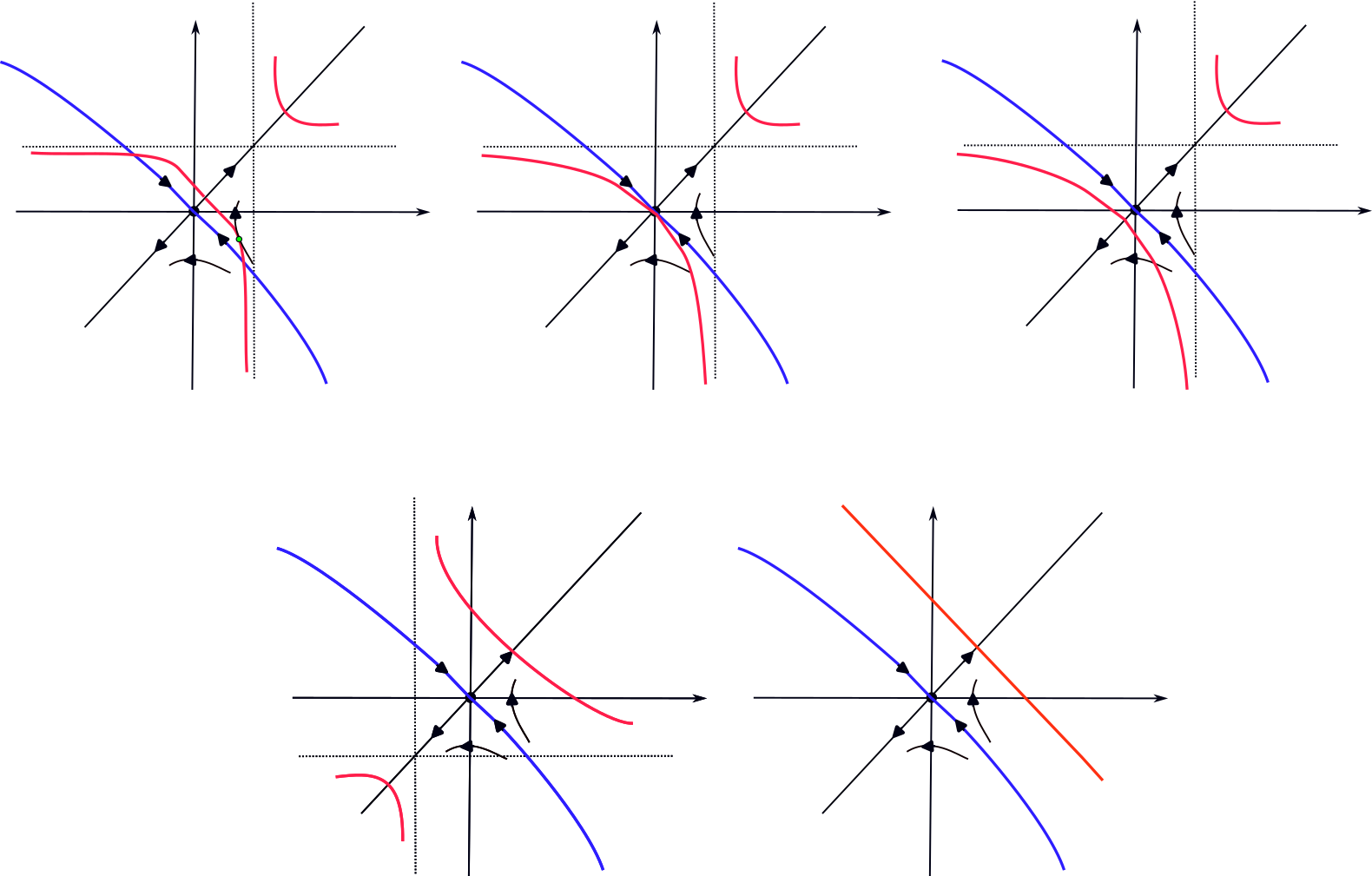}
		{\footnotesize
   \put(-319,120){\footnotesize{(a) $\frac{1}{\gamma_-}<x^*$}}
   \put(-203,120){\footnotesize{(b) $x^*=\frac{1}{\gamma_-}$}}
    \put(-90,120){\footnotesize{(c) $0<x^*<\frac{1}{\gamma_-}$}}
    \put(-248,-10){\footnotesize{(d) $x^*<0$}}
    \put(-133,-10){\footnotesize{(e) $\alpha_+=0$}}
}
         \end{center}
	\caption{The phase portrait of \eqref{3degreepolysys} for $\beta_->0$, $\gamma_->0$ and $\gamma_-^2-4\beta_-<0$, with the curve $\overline \Delta(x,y)=0$ (red). The part of the blue curve located in the fourth quadrant is the graph of $\Pi$. We draw $x=x^*$ and $y=x^*$ using dashed lines. We indicate the contact point $(x,y)=(x_C,-x_C)$ when $x_C$ is positive and contained in the domain of $I$ (Fig. \ref{fig:contactfocus}(a)).}
	\label{fig:contactfocus}
\end{figure}

\begin{theorem}\label{thm-focuscase}
Suppose that $\beta_->0$, $\gamma_->0$ and $\gamma_-^2-4\beta_-<0$. Then the following statements are true.
\begin{enumerate}
     \item{($\alpha_+<0$)} If $\frac{1}{\gamma_-}<x^*$ (Fig. \ref{fig:contactfocus}(a)), then the function $I$ has precisely $1$ zero counting multiplicity on $]0,x^*[$ and, for any sufficiently small $\theta>0$, $\cycl(\cup_{x\in[\theta,x^*-\theta]}\Gamma_x)=2$.
      \item{($\alpha_+<0$)} If $x^*=\frac{1}{\gamma_-}$ (Fig. \ref{fig:contactfocus}(b)), then $I>0$ on $]0,x^*[$ and, for any small $\theta>0$, $\cycl(\cup_{x\in[\theta,x^*-\theta]}\Gamma_x)=1$ (the limit cycle is repelling).
      \item{($\alpha_+<0$)} If $0<x^*<\frac{1}{\gamma_-}$ (Fig. \ref{fig:contactfocus}(c)), then $I>0$ on $]0,x^*[$ and, for any small $\theta>0$, $\cycl(\cup_{x\in[\theta,x^*-\theta]}\Gamma_x)=1$ (the limit cycle is repelling).
        \item{($\alpha_+>0$)} If $x^*<0$ (Fig. \ref{fig:contactfocus}(d)), then $ I<0$ on $]0,\Pi^{-1}(x^*)[$ and, for any small $\theta>0$, $\cycl(\cup_{x\in[\theta,\Pi^{-1}(x^*)-\theta]}\Gamma_x)=1$ (the limit cycle is attracting).
         \item If $\alpha_+=0$ (Fig. \ref{fig:contactfocus}(e)), we have $I<0$ on $]0,\infty[$ and, for any small $\theta>0$, $\cycl(\cup_{x\in[\theta,\frac{1}{\theta}]}\Gamma_x)=1$ (the limit cycle is attracting). 
    
\end{enumerate}
\end{theorem}
\begin{proof}
 We suppose that $\beta_->0$, $\gamma_->0$ and $\gamma_-^2-4\beta_-<0$. Since $V(x)>0$ for all $x\in \mathbb R$, then, for $\alpha_+\ne 0$, we have $hp'(x)<0$ for all $x\ne x^*$ (see \eqref{deriv-hyper}). From  \eqref{contact-important} it follows that $x_C\ge 0$ is well-defined for $x^*\ge\frac{1}{\gamma_-}$. If $x^*=\frac{1}{\gamma_-}$, then $x_C=0$, and, if $x^*>\frac{1}{\gamma_-}$, then $0<x_C<x^*$.

 \textit{Statement 1.} Suppose that $\frac{1}{\gamma_-}<x^*$. From \eqref{deriv-hyper} it follows that $hp(0)>0$. The graph of $hp$ is given in Fig. \ref{fig:contactfocus}(a). The contact points between the orbits of system \eqref{3degreepolysys} and the hyperbola $y=hp(x)$ are $(x,y)=(x_C,-x_C)$ and $(x,y)=(-x_C,x_C)$, with $0<x_C<x^*$.
The domain of $I$ is $[0,x^*[$ (see Fig. \ref{fig:foc-cent}(a)). The proof is similar to the proof of Theorem \ref{thm-distinctnode}.5.  \\
\\
\textit{Statement 2.} Suppose that $x^*=\frac{1}{\gamma_-}$. We have $hp(0)=0$ (see Fig. \ref{fig:contactfocus}(b)). The domain of $I$ is $[0,x^*[$ (see Fig. \ref{fig:foc-cent}(a)). We have $1$ contact point: $(x,y)=(0,0)$. We can show that the graph of $hp$ lies below the graph of $\Pi$ for $x\in ]0,x^*[$ using the same idea as in the proof of Theorem  \ref{theorem-saddle}.2. Then the result follows from Lemma \ref{lemma-below-above}.2 and Theorem \ref{theorem-cyclicity-RHKK}.1.\\
\\
\textit{Statement 3.} Suppose that $0<x^*<\frac{1}{\gamma_-}$. We have $hp(0)<0$. The graph of $hp$ is given in Fig. \ref{fig:contactfocus}(c). There are no contact points between the orbits of system \eqref{3degreepolysys} and $y=hp(x)$.

The domain of $I$ is $[0,x^*[$ (see Fig. \ref{fig:foc-cent}(a)). Again, we can show that the graph of $hp$ lies below the graph of $\Pi$ for $x\in ]0,x^*[$ using the same idea as in the proof of Theorem  \ref{theorem-saddle}.1. Then the result easily follows from Lemma \ref{lemma-below-above}.2 and Theorem \ref{theorem-cyclicity-RHKK}.1.\\
\\
\textit{Statement 4.} Suppose that $x^*<0$. Then $hp(0)>0$ and the graph of $hp$ is given in Fig. \ref{fig:contactfocus}(d). There are no contact points between the orbits of system \eqref{3degreepolysys} and $y=hp(x)$.
The domain of $I$ is $[0,\Pi^{-1}(x^*)[$ (Fig. \ref{fig:foc-cent}(c)). 

Let us prove that the graph of $hp$ lies above the graph of $\Pi$ for $x\in ]0,\Pi^{-1}(x^*)[$. Suppose that there is an intersection between the graph of $hp$ and the graph of $\Pi$ for $x\in ]0,\Pi^{-1}(x^*)[$. This implies the existence of a contact point between the orbits of system \eqref{3degreepolysys} and $y=hp(x)$. This gives a contradiction. Statement $4$ follows now from Lemma \ref{lemma-below-above}.2 and Theorem \ref{theorem-cyclicity-RHKK}.1.\\
\\
\textit{Statement 5.} Assume that $\alpha_+=0$. The graph of  \eqref{line-equation} is given in  Fig. \ref{fig:contactfocus}(e). There are no contact points between the orbits of system \eqref{3degreepolysys} and $y=hp(x)$ because the equation in \eqref{contactpointsCase2} has no solutions.
The domain of $I$ is $[0,+\infty[$ (Fig. \ref{fig:foc-cent}(b)). Again, we can show that the graph of $hp$ lies above the graph of $\Pi$ for $x\in ]0,+\infty[$ (see the proof of Statement $4$). Then Lemma \ref{lemma-below-above}.2 and Theorem \ref{theorem-cyclicity-RHKK}.1 imply the result.

\end{proof}

\appendix

\section{The center case}\label{subsection-centerappen}
\begin{figure}[htb]
	\begin{center}		\includegraphics[width=11.9cm,height=3.3cm]{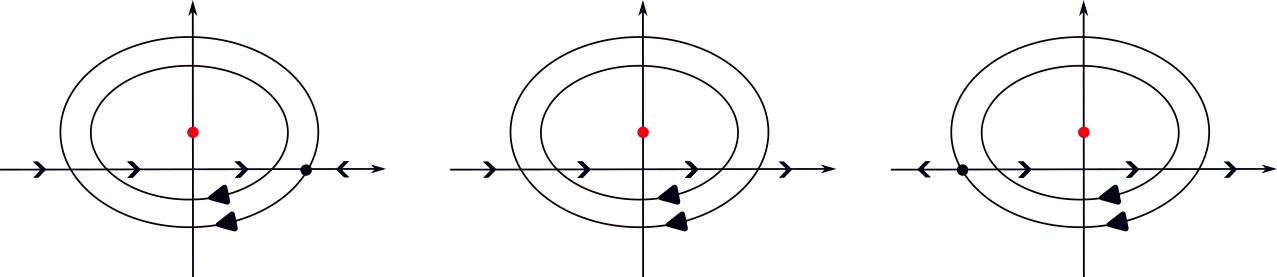}
		{\footnotesize
   \put(-258,27){\footnotesize{$x^*$}}
   \put(-306,-12){\footnotesize{(a) $0<x^*$}}
   \put(-93,27){\footnotesize{ $x^*$}}
    \put(-192,-12){\footnotesize{(b) $\alpha_+=0$}}
     \put(-75,-12){\footnotesize{(c) $x^*<0$}}
}
         \end{center}
	\caption{Phase portraits of $Z^-$ defined in \eqref{linearPWS} and the direction of the sliding vector field \eqref{sliding-linear} along $y=0$, for $\beta_->0$ and $\gamma_-=0$. $Z^-$ has a center.  We do not draw the corresponding phase portraits of $Z^+$.}
	\label{fig:centonly}
\end{figure}
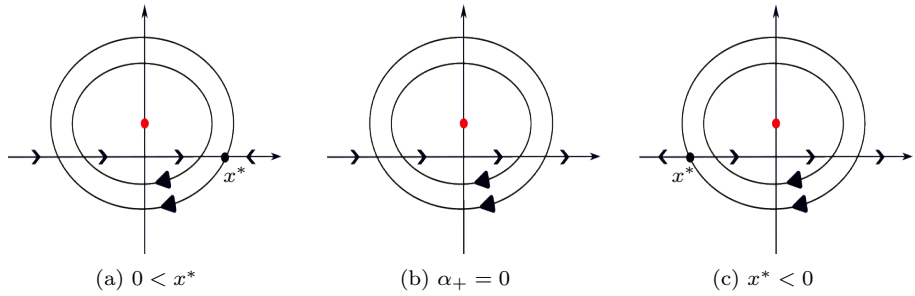
We suppose that $\beta_->0$ and $\gamma_-=0$. System $Z^-$ has a center at $(x,y)=(0,\frac{1}{\beta_-})$.
 We refer to Fig. \ref{fig:centonly}.
The domain and image of $\Pi$ are respectively $[0,+\infty[$ and $]-\infty,0]$. We have $3$ cases.
\begin{enumerate}
    \item[(a)] If $\alpha_+<0$ (hence $x^*>0$), then the domain of $I$ is $[0,x^*[$ (see Fig. \ref{fig:centonly}(a)).
We have $ I>0$ on $]0,x^*[$, and, for any small $\theta>0$,
 $\cycl(\cup_{x\in[\theta,x^*-\theta]}\Gamma_x)=1$. The limit cycle is repelling (Theorem \ref{thm-mainpart1}.2).
    \item[(b)] If $\alpha_+=0$, then the domain of $I$ is $[0,+\infty[$ (see Fig. \ref{fig:centonly}(b)). Since $(\alpha_+,\gamma_-)=(0,0)$, we have $I\equiv 0$.
\item[(c)]  If $\alpha_+>0$ (hence $x^*<0$), then the domain of $I$ is $[0,\Pi^{-1}(x^*)[$ (see Fig. \ref{fig:centonly}(c)). We have $ I<0$ on $]0,\Pi^{-1}(x^*)[$, and, for any small $\theta>0$,
 $\cycl(\cup_{x\in[\theta,\Pi^{-1}(x^*)-\theta]}\Gamma_x)=1$ and the limit cycle is attracting (Theorem \ref{thm-mainpart1}.2).
\end{enumerate}

\begin{remark}
    When $\beta_-<0$ and $\gamma_-=0$, system $Z^-$ has a hyperbolic saddle at $(x,y)=(0,\frac{1}{\beta_-})$. In this case we can find the domain of $I$ in the same way as in Section \ref{saddle-subsection} (see Fig. \ref{fig:sedlo}) and then apply Theorem \ref{thm-mainpart1}.2 when $\alpha_+\ne 0$ (see the center case).
\end{remark}

\section{The case without singularities}\label{nosingularities-subsection}
Here we suppose that $\beta_-=0$ and $\gamma_-\ge 0$. Then system $Z^-$ has no singularities. If $\gamma_->0$, then the line 
$x=\gamma_-y+\frac{1}{\gamma_-}$ is invariant w.r.t. $Z^-$ (see Fig. \ref{fig:nosing}(a)--(d) and Lemma \ref{lemma-LRkappa}.3). The domain and image of $\Pi$ 
are respectively $[0,\frac{1}{\gamma_-}[$ and $]-\infty,0]$. When $\gamma_-=0$, the domain and image of $\Pi$ are respectively $[0,+\infty[$ and $]-\infty,0]$. 
We refer to Fig. \ref{fig:nosing}(e)--(g).

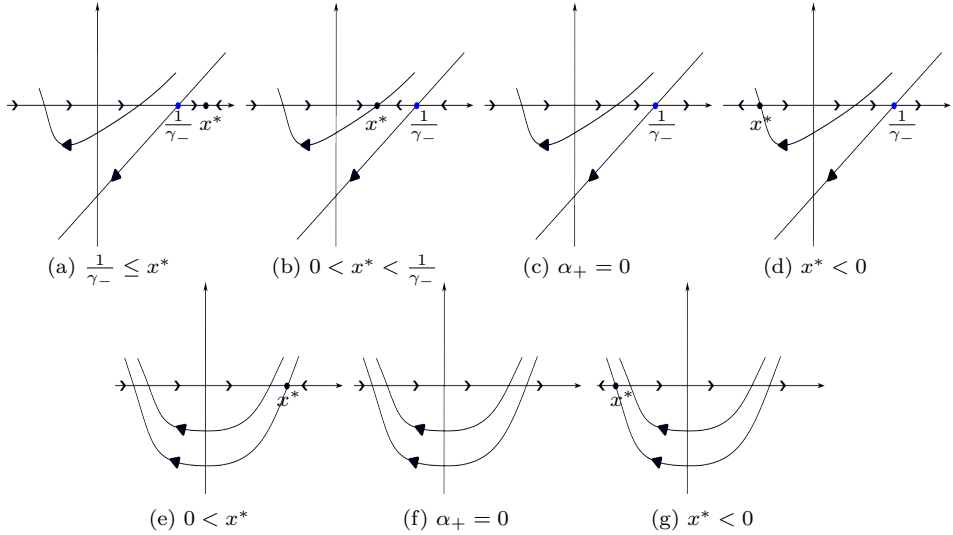
\begin{figure}[htb]
	\begin{center}		\includegraphics[width=12.4cm,height=6.5cm]{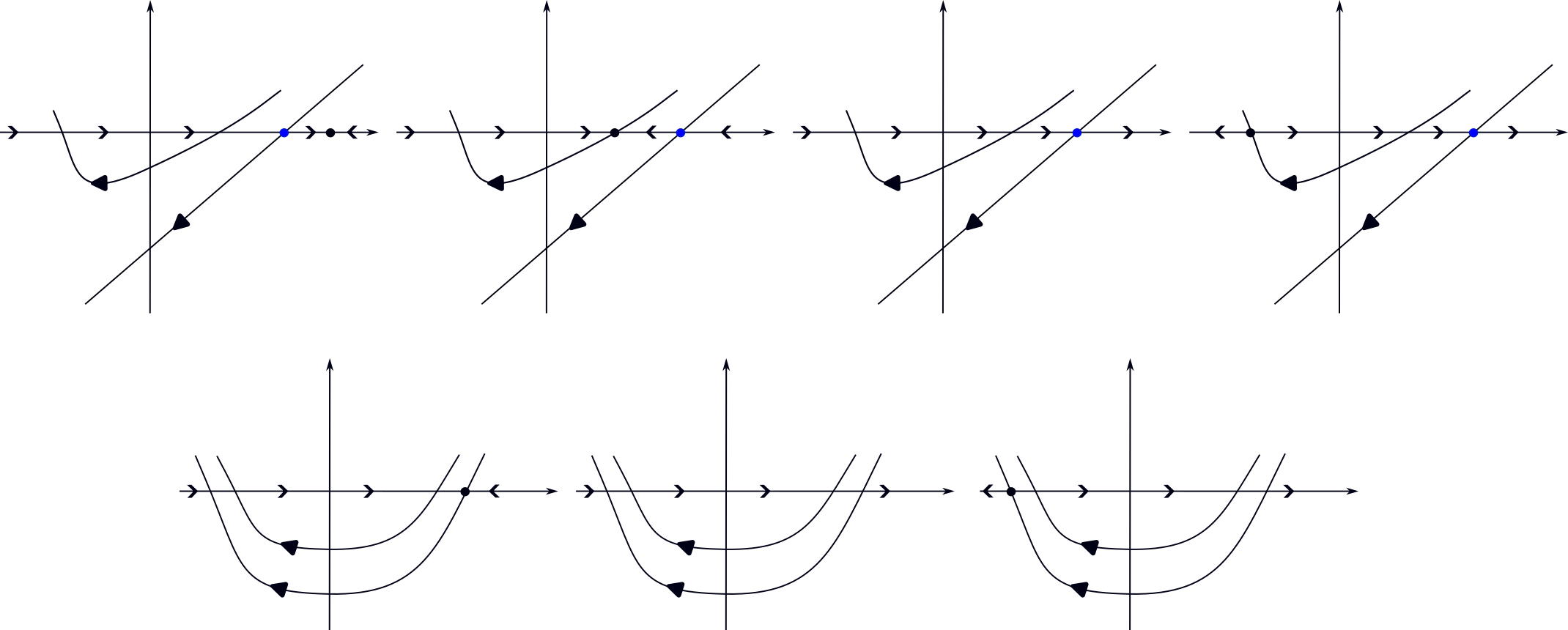}
		{\footnotesize
   \put(-281,137){\footnotesize{$x^*$}}
   \put(-295,137){\footnotesize{$\frac{1}{\gamma_-}$}}
   \put(-338,83){\footnotesize{(a) $\frac{1}{\gamma_-}\le x^*$}}
   \put(-219,137){\footnotesize{$x^*$}}
   \put(-203,137){\footnotesize{$\frac{1}{\gamma_-}$}}
   \put(-254,83){\footnotesize{(b) $0<x^*<\frac{1}{\gamma_-}$}}
   \put(-114,137){\footnotesize{$\frac{1}{\gamma_-}$}}
   \put(-160,83){\footnotesize{(c) $\alpha_+=0$}}
   \put(-75,137){\footnotesize{$x^*$}}
   \put(-24,137){\footnotesize{$\frac{1}{\gamma_-}$}}
   \put(-70,83){\footnotesize{(d) $x^*<0$}}
   \put(-252,33){\footnotesize{$x^*$}}
   \put(-300,-12){\footnotesize{(e) $0<x^*$}}
   \put(-130,33){\footnotesize{ $x^*$}}
    \put(-205,-12){\footnotesize{(f) $\alpha_+=0$}}
     \put(-112,-12){\footnotesize{(g) $x^*<0$}}
}
         \end{center}
	\caption{Phase portraits of $Z^-$, with $\beta_-=0$, defined in \eqref{linearPWS} and the direction of the sliding vector field \eqref{sliding-linear} along $y=0$. (a)--(d) $\gamma_->0$. (e)--(g) $\gamma_-=0$.  We do not draw the corresponding phase portraits of $Z^+$.}
	\label{fig:nosing}
\end{figure}
 We have
\begin{enumerate}
    \item[(a)] If $\gamma_->0$ and $\frac{1}{\gamma_-}\le x^*$, then the domain of $I$ is $[0,\frac{1}{\gamma_-}[$ (see Fig. \ref{fig:nosing}(a)). 
When $\frac{1}{\gamma_-}< x^*$, from Theorem \ref{thm-mainpart1}.1 it follows that 
$ I<0$ on $]0,\frac{1}{\gamma_-}[$. For any small $\theta>0$, 
 $\cycl(\cup_{x\in[\theta,\frac{1}{\gamma_-}-\theta]}\Gamma_x)=1$ and the limit cycle is attracting. When $x^*=\frac{1}{\gamma_-}$, then $I\equiv 0$ (see the paragraph after \eqref{condition-theorem-imp}).
    \item[(b)] If $\gamma_->0$ and $0<x^*<\frac{1}{\gamma_-}$, then the domain of $I$ is $[0,x^*[$ (see Fig. \ref{fig:nosing}(b)). Theorem \ref{thm-mainpart1}.1 implies that 
$ I>0$ on $]0,x^*[$, and, for any small $\theta>0$,
 $\cycl(\cup_{x\in[\theta,x^*-\theta]}\Gamma_x)=1$ (the limit cycle is repelling).
\item[(c)]  If $\gamma_->0$ and $\alpha_+=0$, then we have the same domain of $I$ as in the case (a) (see Fig. \ref{fig:nosing}(c)). 
Again, Theorem \ref{thm-mainpart1}.1 implies that 
$ I<0$ on $]0,\frac{1}{\gamma_-}[$.
\item[(d)]  If $\gamma_->0$ and $x^*<0$, then the domain of $I$ is $[0,\Pi^{-1}(x^*)[$ (see Fig. \ref{fig:nosing}(d)). We have $ I<0$ on 
$]0,\Pi^{-1}(x^*)[$ and, for any small $\theta>0$,
 we have $\cycl(\cup_{x\in[\theta,\Pi^{-1}(x^*)-\theta]}\Gamma_x)=1$ and the limit cycle is attracting (Theorem \ref{thm-mainpart1}.1).
 \item[(e)] If $\gamma_-=0$ and $0< x^*$, then the domain of $I$ is $[0,x^*[$ (see Fig. \ref{fig:nosing}(e)). 
We have $ I>0$ on $]0,x^*[$, and, for any small $\theta>0$,
 $\cycl(\cup_{x\in[\theta,x^*-\theta]}\Gamma_x)=1$ and the limit cycle is repelling (Theorem \ref{thm-mainpart1}.1).
\item[(f)]  If $\gamma_-=0$ and $\alpha_+=0$, then the domain of $I$ is $[0,+\infty[$ (see Fig. \ref{fig:nosing}(f)). We have $I\equiv 0$.
\item[(g)]  If $\gamma_-=0$ and $x^*<0$, then the domain of $I$ is $[0,\Pi^{-1}(x^*)[$ (see Fig. \ref{fig:nosing}(g)). We have the same sign of $I$ as in the case (d).
\end{enumerate}

\section*{Declarations}
 
\textbf{Ethical Approval} \ 
Not applicable.
 \\
\\
 \textbf{Competing interests} \  
The authors declare that they have no conflict of interest.\\
 \\
\textbf{Authors' contributions} \  All authors conceived of the presented idea, developed the theory, performed the computations and
contributed to the final manuscript.  \\ 
\\ 
\textbf{Availability of data and materials}  \
Not applicable.

\bibliographystyle{plain}
\bibliography{bibtex}
\end{document}